\documentclass{article}
\usepackage{graphicx}
\usepackage{amsthm,amssymb,amscd,amsmath,latexsym,indentfirst,color,amsfonts}
\usepackage[mathcal]{eucal}
\usepackage{geometry,amsthm,graphics,tabularx,shapepar,float}
\usepackage{tikz-cd}
\usetikzlibrary{cd}
\usepackage{stackrel}
\usepackage{graphicx}
\usepackage{extarrows}
\usepackage[all]{xypic}
\usepackage{tikz}
\usepackage{bm}
\usepackage{extarrows}
\theoremstyle{plain}
\newtheorem{theorem}{Theorem}[section]
\newtheorem{prop}[theorem]{Proposition}
\newtheorem{coro}[theorem]{Corollary}
\newtheorem{lemma}[theorem]{Lemma}
\newtheorem{ex}[theorem]{Example}

\newtheorem{rem}[theorem]{Remark}

\newtheorem{defi}[theorem]{Def.}

\newcommand{\ble}{\begin {lemma}}
\newcommand{\ele}{\end {lemma}}
\newcommand{\betahm}{\begin {theorem}}
\newcommand{\ethm}{\end {theorem}}
\newcommand{\bco}{\begin {coro}}
\newcommand{\eco}{\end {coro}}
\newcommand{\bex}{\begin {ex}}
\newcommand{\eex}{\end {ex}}
\newcommand{\be}{\begin {equation}}
\newcommand{\ee}{\end {equation}}
\newcommand{\bp}{\begin {proof}}
\newcommand{\ep}{\end {proof}}
\newcommand{\rt}{\rightarrow}
\begin{document}
\title{Deformed Laurent series rings and completions of the Weyl division ring }

\author{Gang Han %\thanks{The author was supported by NSFC 12071422}
	\\School of Mathematics, Zhejiang
	University\\mathhgg@zju.edu.cn \\[2mm]	
	Yulin Chen
	\\School of Mathematics, Zhejiang
	University\\tschenyl@163.com
	\\[2mm]	
	Zhennan Pan
	\\School of Mathematics, Zhejiang
	University\\pzn1025@163.com
}

\date{Jan. 26, 2024}

\maketitle
\begin{abstract}
	 Let $ L((T^{-1}))$ be the space of (inverse) Laurent series
	with coefficients in some field $L$. It has a  standard degree map and the induced topology. With its usual addition and a new product on this space  which  is continuous and preserves the  standard degree map,  it will be  a complete topological division ring, and called a  deformed Laurent series  ring. Under mild restrictions, we give  the necessary and sufficient conditions for a product on $ L((T^{-1}))$ to make it a  deformed Laurent series  ring. Then we apply the above theory to construct the completions of the Weyl division ring $D_1$, over some field of characteristic 0, with respect to a class of discrete valuations on it. Such completions are topological division rings with
	nice properties. For instance, their valuation rings are non-commutative Henselian rings; the centralizer of each element not in the center is commutative.

\end{abstract}


	\tableofcontents
	\section{Introduction}
	 \setcounter{equation}{0}\setcounter{theorem}{0}
	
	Let $K$ be a field of Characteristic 0. Let $A_1(K)$ be the first Weyl algebra over $K$ generated by 2 generators $\textbf{p},\textbf{\textbf{q}}$ satisfying the relation  \be\label{f17}[\textbf{q},\textbf{p}]=\textbf{qp}-\textbf{pq}=1.\ee We adopt the convention (\ref{f17}) following Joseph \cite{jo}, from which we benefit a lot. The algebra $A_1(K)$ is a simple Noetherian domain of GK dimension 2. The quotient division algebra $ D_1(K)$ of $A_1(K)$ is usually  called the  Weyl division ring or Weyl skew field.% ( Elements in $D_1(K)$ are equivalence classes of  $zw^{-1}$ with $z,w\in A_1(K), w\ne0$. )

People have been interested in $A_1(K)$ and $D_1(K)$ for quite some time, and have conducted extensive research in this area. For example in \cite{ml} Makar-Limanov showed that  $D_1(K)$ contains a two-generator free subalgebra.

There were also some research about the valuations on  $A_1(K)$ and $D_1(K)$. 	A class of discrete valuations of $D_1(K)$ compatible with the Bernstein
	filtration  were classified in 	\cite{w}. All the valuations on the real Weyl algebra
	$A_1(\mathbb R)$ whose residue field is $\mathbb R$ are classified in \cite{v}.
	
		A map
	$u:D_1(K)\rt \mathbb{R}_{-\infty}$ such that $-u:D_1(K)\rt \mathbb{R}_{\infty}, f\mapsto -u(f)$ is a rank 1 valuation will be called a {degree map} on $D_1(K)$. 	The degree maps $u$ on $A_1(K)$ over $K$ with the usual basis $\{\textbf{p}^i\textbf{q}^j|i,j\ge  0\}$ being an evaluation basis for  $u$ are classified in Proposition \ref{f28}. It is interesting that they are in   1-1  correspondence with the set  $\{(\rho,\sigma)\in \mathbb R^2|\rho+\sigma\ge0\}$. Up to equivalence, those degree maps $\mathsf{v}_{\rho,\sigma}$ with $(\rho,\sigma)\in \mathbb Z^2$, $\rho$ and $\sigma$ are coprime, and $\rho+\sigma\ge 0$, correspond to discrete valuations. See Section 2.	
	
	In the influential paper \cite{d}, Dixmier studied the structure of $A_1(K)$ in great detail, and in the end he posed six problems. The first problem  asks if
	every endomorphism of the Weyl algebra $A_1(K)$ is an automorphism. The conjecture that its answer is positive was later called Dixmier Conjecture. One approach to study this conjecture is to find all the solutions (up to automorphisms of  $A_1(K)$) of the equation $[x,y]=1$ for $x,y$ in $A_1(K)$. We believe that it is better to study this equation in $D_1(K)$ and its completions.
		One  purpose of this note is to construct the completions of $D_1(K)$ with respect to such discrete valuations. In fact, certain skew
	Laurent series algebras and the algebra of formal 
	pseudo-differential operators are  completions of $D_1(K)$ with respect to different valuations.  Since there are infinitely many nonequivalent  degree maps on $D_1(K)$, it will be interesting to find the relationship between different completions.
	
	To construct such completions, we did the following. Let $ L((T^{-1}))$ be the space of (inverse) Laurent series
	with coefficients in some field $L$, which has a  standard degree map and the corresponding topology.
	We define a new product on this space (while the addition is as usual)  which  is continuous and preserves the  standard degree map. We call it a  deformed Laurent series ring. See Def. \ref{e39}.  Assume that it satisfies
	\[ T\cdot a=aT+\sum_{i=0}^{-\infty}\delta_i(a)T^i, a\in L,\] where $\delta_i:L\rightarrow L, i\in \mathbb Z_{\le 0} ,$ are additive maps. 	
	Then the necessary and sufficient condition for it to make  $ L((T^{-1}))$ a  deformed Laurent series ring is (See Theorem \ref{f22})
	
		\[  \delta_i(ab)= a\delta_i(b)+\delta_i(a)b+ \sum_{k> 0}\sum_{ (l,j_1,\cdots,j_k)
	}^{j_1+\cdots+j_k-k+l=i}\binom{l}{k}\delta_l(a)\delta_{j_1}\cdots \delta_{j_k}(b), a,b\in L,\] where the summation is over all $l,j_1,\cdots,j_k\le 0$ and $j_1+\cdots+j_k-k+l=i$. The proof uses a lot of combinatorial identities.
	
	 This ring, denoted by 	$L((T^{-1};\Tilde{\delta}))$, is a complete topological skew field. It is a division ring of formal Laurent series with a deformed multiplication. Its center  is the field $F=\cap_{i\le 0} \text{Ker}(\delta_i)$, and such $\delta_i$ are $F$-linear differential operators on $L$. 
	 
	 The proof of above results is in Section 3, in which a symmetric version of such  results is also given.
	
Let $\widehat{D}_{r,s}$ be the	completion  of $D_1(K)$ with respect to $ \mathsf{v}_{r,s}$, where $(r,s )\in \mathbb Z^2$ with $r+s> 0$. To construct $\widehat{D}_{r,s}$, we need to first construct  $\widetilde{D}_{r,s}$  and $\check{D}_{r,s}$ using the theory of deformed Laurent series ring, which are also complete skew fields. Then we embed $D_1(K)$ into it. See Proposition \ref{fp1} and Corollary \ref{fp2}. Non-commutative Henselian rings are defined and studied in \cite{a}. The respective valuation rings of  $\widetilde{D}_{r,s}$, $\check{D}_{r,s}$, and $\widehat{D}_{r,s}$  are in fact non-commutative Henselian rings.
See Proposition \ref{f35}.

 In Section 5, a pair of topological generators of $\widehat{D}_{r,s}$ is given.

In \cite{b}, a noncommutative algebra $A$ is said to satisfy the commutative centralizer
condition (the ccc, for short) if the centralizer of each element not in the center of $A$ is a commutative algebra.  Bavula found that many algebras, including $D_1(K)$, satisfy the ccc  \cite{b}. As in \cite{b} one can show that $\widetilde{D}_{r,s}$  and $\check{D}_{r,s}$, thus $\widehat{D}_{r,s}$,  also satisfy the ccc. See Section 6.\\[1mm]
	
Here are some of the conventions of the paper.

	All the rings and algebras are associative and unital, with the multiplicative identity $1$. All the ring (and algebra) homomorphisms are unital. 
	
	For a ring $R$, $R^\times$ is the set of nonzero elements in $R$.

	For $a,b\in \mathbb R$, let
	\begin{equation}
		[a,b]_k=
		\left\{
		\begin{array}{ll}
			\prod_{i=0}^{k-1} (a-ib),  &k> 0;\\
			1, & k=0.
		\end{array}
		\right.
	\end{equation}
	
	For $k\ge 0$, let $[a]_k=[a,1]_k$. Then $[a]_k=a(a-1)\cdots(a-k+1)$ for $k\ge 1$, and $ [a]_0=1 $.
	
For $a\in \mathbb R, k\ge0$,	$\binom{a}{k}=[a]_k/k!$.

		$x,y,z,w$ are elements in $A_1(K)$, where 	$K$ is a field.

	For $f\in K(X), i\ge 0$, $f^{(i)}$ denotes the $i$-th derivative of $f$.
	
	$X,Y,T$ are indeterminants.
	
	$char\ L$ denotes the characteristic of the field $L$.
	
	$\mathbb{R}_{\infty}=\mathbb{R}\cup \{\infty\}$;		$\mathbb{R}_{\ge 0}=\{a\in \mathbb{R}|a\ge 0\}$; and similar notations will be used.
		
%	$\Lambda_+=\{(x,y)\in \mathbb{Z}^2|x+y>0\}$

 %$\rho,\sigma,\tau$ denote real numbers;
	
	For integers $m$ and $n$, $gcd\{m,n\}$ is the non-negative greatest common divisor of $m$ and $n$.

	"iff" stands for "if and only if".
	
		All the indices \textbf{under} the sign $\sum$ will be assumed to be integers in $\mathbb Z_{\le 0}$, unless otherwise stated. For example, $\sum_{(j_1,\cdots,j_k)}f(j_1,\cdots,j_k )$ means to sum $f(j_1,\cdots,j_k )$ over all $(j_1,\cdots,j_k)$ with $j_i\in \mathbb Z_{\le 0}$ for $ i=1,\cdots,k$. Here is another example,
	
	\be\label{e43}\sum_{(j_1,\cdots,j_k)}^{j_1+\cdots+j_k-k=1} f(j_1,\cdots,j_k )\ee means to make summations of $f(j_1,\cdots,j_k )$ over all $(j_1,\cdots,j_k)\in (\mathbb Z_{\le 0})^{k}$ satisfying
	$j_1+\cdots+j_k-k=1$.
	
%\section*{Acknowledgements}
\centerline{\textbf{Acknowledgements}}

We would like to heartily thank Binyong Sun for his advice to improve the results. We are very thankful to  Dongwen Liu, Fangyang Tian,  Chengbo  Wang,  Shilin Yu, Tao Feng and Yeshun Sun for 
helpful discussions during the research. 
We are grateful to	Qinghu Hou at Nankai University, who told us Lemma \ref{fl1} and its proof.
	
	\section{A class of degree maps on $D_1$}
	
	 \setcounter{equation}{0}\setcounter{theorem}{0}
	Let	$R$ be a ring. 
	
	\begin{defi}\label{}
	 A map $v:R\rt \mathbb{R}_{\infty}$ is called a rank 1 valuation on $R$ if it satisfies 
	\begin{enumerate}		
		\item $v(fg)=v(f)+v(g), \forall f,g\in R $;
		\item $v(f+g)\ge min\{v(f),v(g)  \},\forall f,g\in R$;
		\item $v(f)=\infty$ iff $f=0$.
	
	\end{enumerate}
	
	If $v(R^\times)=0$ then the valuation $v$ is called trivial; 
	If $v(R^\times)\cong \mathbb Z$ then the valuation $v$ is called discrete.
	
		Two rank 1 valuations $v_1,v_2$ on $R$ are called equivalent if there exists some $c\in\mathbb R_{>0}$, $v_2(f)=cv_1(f)$ for all $f\in R^\times$.
	\end{defi}

All the valuations and degree maps considered in this paper will be assumed to be rank 1.

	A map
	$u:R\rt \mathbb{R}_{-\infty}$ such that $-u:R\rt \mathbb{R}_{\infty}, f\mapsto -u(f)$ is a rank 1 valuation will be called a \textbf{degree map} on $R$, i.e., $u$ is a degree map on $R$ iff $u$ satisfies 
		\be \label{e12}
	\begin{array}{lll}
		(1)\ u(fg)=u(f)+u(g), \forall f,g\in R;	\\		
		(2)\ u(f+g)\le max\{u(f),u(g)  \},\forall f,g\in R;\\		
     	(3)\ u(f)=-\infty\ \text{iff} \ f=0.
	\end{array}
\ee

	 A degree map $u$ is called trivial (resp. discrete) if so is the valuation $-u$. Two degree maps $u_1,u_2$ on $R$ are called equivalent if so are $-u_1,-u_2$.

	For a degree map $u$ on $R$, there exists a metric  on $R$ defined by $d(f,g)=e^{u(f-g)}, f,g\in R.$ Given a degree map $u$ on $R$, we will always equip $R$ with the topology induced by this metric.
	
Let  $K$ be a field of characteristic 0.	Assume that $A$ is an algebra over $K$.
\begin{defi}
 A valuation $v:A\rt \mathbb R_{\infty}$ is said to be over $K$,  or $v$ is a $K$-valuation, if $v|_K$ is the trivial valuation on $K$. Similarly,  A degree map $u:A\rt \mathbb R_{-\infty}$ is said to be over $K$, or $u$ is a $K$-degree map, if $-u$ is a  valuation of $A$ over $K$.
\end{defi}	
 Let $\mathcal{C}=\{y_i|i\in I\}$ be a $K$-basis of $A$. $\mathcal{C}$ is called an evaluation basis for the $K$-degree map $u$ on $A$, if \[ u(\sum_{i\in I} \lambda_i y_i)=sup\{u(y_i)| i\in I, \lambda_{i}\ne0\},\ \sum_{i\in I} \lambda_i y_i\in A^\times.\] See Definition 2 of \cite{kkl}.

Assume  $u:A\rt \mathbb Z_{-\infty}$ is a discrete degree map over $K$. Let $A^{\le i}=u^{-1}(\mathbb Z_{\le i}\cup \{-\infty\})$. Then by (\ref{e12}) one has $A^{\le i}$ is a $K$-vector space for each $i$, and $A^{\le i}\cdot A^{\le j}\subseteq A^{\le i+j}$. So one has a graded $K$-algebra \[gr(A)=\oplus_{i\in \mathbb Z}\ gr_i(A), gr_i(A)=A^{\le i}/A^{\le i-1},\] called the associated graded algebra of $A$ with respect to $u$.\\[2mm]

 For simplicity, denote the Weyl algebra $A_1(K)$ and the Weyl division ring $D_1(K)$ by $A_1$ and $D_1$ respectively. 
 
	 One knows that $\mathcal{B}=\{\textbf{p}^i\textbf{q}^j|i,j\in \mathbb Z_{\ge  0}\}$ is a basis of $A_1$. For 
	$ z=\sum {\lambda_{i,j}}\textbf{p}^i\textbf{q}^j\in A_1\setminus\{0\}$, let $supp\{z\} = \{(i,j)\in (\mathbb{Z}_{\ge  0})^2|\lambda_{i,j}\neq 0\}, supp\{0\}=\o$.
	
	The formula for the product of the basis elements is 	\be \label{f32} \textbf{p}^i\textbf{q}^j\cdot \textbf{p}^s\textbf{q}^t=	\textbf{p}^{i+s}\textbf{q}^{j+t}+\sum_{k\ge1}\frac{[j]_k[s]_k}{k!}\textbf{p}^{i+s-k}\textbf{q}^{j+t-k}.\ee
	
	Assume $(\rho,\sigma)\in \mathbb R^2$. The map \[\mathbf{v}_{\rho,\sigma}:A_1\rt \mathbb R_{-\infty}, \ \ 0\mapsto -\infty;  \textbf{p}^i\textbf{q}^j\mapsto \rho i+\sigma j, 0\ne\sum \lambda_{i,j}\textbf{p}^i\textbf{q}^j\mapsto sup \{\mathsf{v}_{\rho,\sigma}(\textbf{p}^i\textbf{q}^j)|  \lambda_{i,j}\ne0\} \]  is originally defined in \cite{d} in the case $(\rho,\sigma)\in \mathbb R^2\setminus (0,0)$ and $\rho+\sigma>0$.

		\ble For  $(\rho,\sigma)\in \mathbb R^2\setminus (0,0)$ and $\rho+\sigma\ge0$, 
	$\mathsf{v}_{\rho,\sigma}:A_1\rt \mathbb R_{-\infty}$ is a degree map.
	\ele

	\bp
	%Firstly, we prove that $-\mathsf{v}_{\rho,\sigma}:A_1\to \mathbb{R}_{\infty}$ is a valuation.  
	It  follows from (\ref{f32}) that $\mathsf{v}_{\rho,\sigma}(\textbf{p}^i\textbf{q}^j\cdot \textbf{p}^s\textbf{q}^t)=	\mathsf{v}_{\rho,\sigma}(\textbf{p}^i\textbf{q}^j)+	\mathsf{v}_{\rho,\sigma}(\textbf{p}^s\textbf{q}^t),\  \textbf{p}^i\textbf{q}^j,\textbf{p}^s\textbf{q}^t\in \mathcal{B}$. In fact in Lemma 2.4 of \cite{d}, it is shown that
	$\mathsf{v}_{\rho,\sigma}(zw)=	\mathsf{v}_{\rho,\sigma}(z)+	\mathsf{v}_{\rho,\sigma}(w)\ (\forall z,w \in A_1)$ if $\rho+\sigma>0$, which can be proven similarly in the case $\rho+\sigma=0$. Thus (\ref{e12}) (1) holds.

	As $supp\{z+w\}\subseteq supp\{z\}\cup supp\{w\}$, one has $\mathsf{v}_{\rho,\sigma}(z+w) \leq \max\{\mathsf{v}_{\rho,\sigma}(z),\mathsf{v}_{\rho,\sigma}(w)\}$, which is  (\ref{e12}) (2).
	It is clear that $\mathsf{v}_{\rho,\sigma}$ also satisfies (\ref{e12}) (3).
	Therefore $\mathsf{v}_{\rho,\sigma}$ is a degree map on $A_1$.
	
		\ep

\begin{prop}\label{f28}
(1) The degree maps on $A_1$ over $K$  with $\mathcal{B}$ being an evaluation basis are exactly those $\mathsf{v}_{\rho,\sigma}$ with $ \rho+\sigma\ge0$. See figure 1. %Thus the degree maps on $D_1(K)$ over $K$  are in 1-1  correspondence with  $\{(\rho,\sigma)\in R^2|\rho+\sigma\ge0\}$. 

(2) Two  degree maps $\mathsf{v}_{\rho_i,\sigma_i} (i=1,2)$ on $A_1$ are equivalent iff there exists some $r\in \mathbb R_{>0}$, $ (\rho_2,\sigma_2)=r(\rho_1,\sigma_1)$.
 
(3) $\mathsf{v}_{\rho,\sigma}$ is the trivial degree map on $A_1$ iff $(\rho,\sigma)=(0,0)$.
	
	The degree map $\mathsf{v}_{\rho,\sigma}$, where $\rho+\sigma\ge0$ and $(\rho,\sigma)\ne (0,0)$, is discrete  iff $\rho,\sigma$ are $\mathbb Q$-dependent. In this case, up to equivalence one can choose $(\rho,\sigma)$ to be coprime integers.

\end{prop}	
\begin{center}
	\begin{tikzpicture}
	\draw [very thin, color=gray](-1.7,1.9)  -- (2,1.9);
	\draw [very thin, color=gray](-0.7,1)  -- (2,1) node[above right] {$\rho+\sigma \geq 0$};
	\draw [very thin, color=gray](1.3,-1)  -- (2,-1);
	\draw[->,very thick] (-1.9,0) -- (2.1,0) node[right] {$\rho$};
	\draw[->,very thick] (0,-1.9) -- (0,2.1) node[above] {$\sigma$};
	\draw[very thick](-1.9,1.9) -- (1.9,-1.9);

\end{tikzpicture}
\end{center}
\begin{center}
	figure 1
\end{center}

\bp
(1) Let $u:A_1\rt \mathbb R_{-\infty}$ be  a degree map. Assume $u(\textbf{p})=\rho, u(\textbf{q})=\sigma$. Then by (\ref{e12}) (1), $u(\textbf{p}^i\textbf{q}^j)=\rho i+\sigma j, i,j\ge 0$; by (\ref{e12}) (2), for $0\ne z=\sum \lambda_{i,j}\textbf{p}^i\textbf{q}^j\in A_1$, $u(z)=sup\{\mathsf{v}_{\rho,\sigma}(\textbf{p}^i\textbf{q}^j)|  \lambda_{i,j}\ne0\}$. Thus $u=\mathsf{v}_{\rho,\sigma}$.

As $1=\textbf{qp}-\textbf{pq}$, by (\ref{e12}) (2), \[0=\mathsf{v}_{\rho,\sigma}(1)\le max\{\mathsf{v}_{\rho,\sigma}(\textbf{qp}), \mathsf{v}_{\rho,\sigma}(\textbf{pq})\}=\rho+\sigma. \] Thus one has $\rho+\sigma\ge0$.

Conversely, for any $(\rho,\sigma)$ with $\rho+\sigma\ge0$, $\mathsf{v}_{\rho,\sigma}$ is a degree map on $A_1$ by last lemma.

(2) The 'if' part is clear. Now we prove 'only if'. Let $\mathsf v_i=\mathsf{v}_{\rho_i,\sigma_i}, i=1,2$. As $\mathsf v_1,\mathsf v_2$ are equivalent, there exists some $r>0$, $\mathsf v_2=r\mathsf v_1$. Thus \[(\rho_2,\sigma_2)=(\mathsf v_2(\textbf{p}),\mathsf v_2(\textbf{q}))=(r\mathsf v_1(\textbf{p}),r\mathsf v_1(\textbf{q}))=r(\rho_1,\sigma_1).\]

 (3) This follow from $\mathsf{v}_{\rho,\sigma}(A_1^\times)=\rho \mathbb Z_{\ge0}+\sigma \mathbb Z_{\ge0}.$

\ep
		
		As $A_1$ is an Ore domain, a degree map $u:A_1\rt \mathbb R_{-\infty}$ can be extended to 
	\[u:D_1\rt \mathbb R_{-\infty}, zw^{-1}\mapsto u(z)-u(w)\] (cf. Proposition 9.1.1 of \cite{co}), and a  degree map on $D_1$ restricts to a  degree map on $A_1$. Thus there is a 1-1 correspondence between degree maps on  $A_1$ and degree maps on $D_1$. We will regard $\mathsf{v}_{\rho,\sigma}$ as degree maps on $D_1$. 
	
	\bex
	
	1. $(\rho,\sigma)=({1,-1})$. %$v_{1,-1}:D_1\rt \mathbb Z\cup -\infty$.
	 Consider the skew
	Laurent series algebra $L((T^{-1};\gamma)), L=K(\alpha)$ with $ \gamma\in \text{Aut}_K(L), \gamma(\alpha)=\alpha-1, T f=\gamma(f) T\ for\ f\in L$ (cf. \cite{b}). There is an embedding
	\[D_1\rt L((T^{-1};\gamma)), \textbf{p}\mapsto T, \textbf{q}\mapsto T^{-1}\alpha\ (\text{thus\ }\textbf{pq}\mapsto \alpha) ,\] and $L((T^{-1};\gamma))$ is the  completion of $D_1$ with respect to the topology induced by the  degree map $\mathsf{v}_{1,-1}$. 
	
	2.  $(\rho,\sigma)=({0,1})$. Consider the algebra of formal 
	pseudo-differential operators  $L((T^{-1};\delta)), L=K(\alpha)$,   with $ \delta\in \text{Der}_K(L), \delta(f)=f'$ for $ f\in L$. One has $T f=fT+\delta(f)$. See \cite{g}.
	 There is an embedding
	\[D_1\rt L((T^{-1};\delta)), \textbf{p}\mapsto \alpha, \textbf{q}\mapsto T,\]  and $L((T^{-1};\delta))$ is the  completion of $D_1$ with respect to the topology induced by the  degree map $\mathsf{v}_{0,1}$. 
	 
	\eex
	
In this paper we will construct the  completions of $D_1$ with respect to the topology induced by the  degree map $\mathsf{v}_{\rho,\sigma}$, where $(\rho,\sigma )\in \mathbb Z^2$ with $\rho+\sigma> 0$.

	\section{Deformed Laurent series rings with the standard degree map }
	 \setcounter{equation}{0}\setcounter{theorem}{0}
	Let $L$ be a field, and  $L((T^{-1}))=\{\sum_{i\le k}a_iT^i|a_i\in L, k\in \mathbb Z\}$ be the left $L$-vector space of (inverse) Laurent series. We temporarily forget its usual multiplication. %We view $L$ as a subset of $L((T^{-1}))$ by $a=aT^0, a\in L,$ and $T^0=1$.
	
	The map
	\be\label{e40}deg:L((T^{-1}))\rt \mathbb \mathbb Z_{-\infty}, \ \ 0\mapsto -\infty, L\setminus \{0\}\mapsto 0, \sum_{i\le k}a_iT^i\mapsto  k\ (a_k\ne 0),\ee satisfies 
	(2) and (3) of (\ref{e12}), 
	and we call it the standard degree map on $L((T^{-1}))$.  Equip $L((T^{-1}))$ with the  topology induced by the metric \[d(f,g)=e^{deg(f-g)}, f,g\in L((T^{-1})).\]
	% basis \[ V(z, n)=\left \{ y \in L((T^{-1}))|deg(y-z)<n  \right \} , n\in \mathbb Z_{<0}, \ z\in L((T^{-1})).\]

	\begin{defi}  \label{e39}
		Let $L((T^{-1}))$  have  the usual addition, the standard degree map and the corresponding topology.
		Assume  $L((T^{-1}))$ has a   multiplication \[\Tilde{\delta}:L((T^{-1}))\times L((T^{-1}))\rt L((T^{-1})), (f,g)\mapsto f\cdot g\] that satisfies
		\begin{enumerate} 
				\item the multiplication $\Tilde{\delta}$ is associative, and distributive over addition ; $\Tilde{\delta}$ is continuous;
			\item $\Tilde{\delta}$ agrees with the left multiplication of $L$ on $L((T^{-1}))$,i.e., $b\cdot\sum_{i\le k}a_iT^i=\sum_{i\le k}(ba_i)T^i$;
			
				\item $T^i\cdot T^j=T^{i+j}, i,j\in \mathbb Z;$ ($T^0=1$);
			
					\item   $deg(f\cdot g)=deg(f)+deg(g), f,g\in L((T^{-1}))$.
		\end{enumerate}
		
		Then $L((T^{-1}))$ is a  ring, and $deg$ satisfies all the 3 axioms in (\ref{e12}) thus is a degree map on it. We denote it by $L((T^{-1};\Tilde{\delta}))$.
	\end{defi}

	As 	$1\cdot aT^i=aT^i$ and $ aT^i\cdot 1=a(T^i\cdot 1)=aT^i$ by the axioms, the multiplicative identity 1 of $L$ is also the multiplicative identity of $L((T^{-1};\Tilde{\delta}))$, and $L((T^{-1}))$ contains $L$ as a subring.

\ble
The ring	$L((T^{-1};\Tilde{\delta}))$ is a topological  division ring  which is  complete with respect to the topology induced by the degree map $deg$.
\ele

\bp It is clear that the addition and multiplication of $L((T^{-1};\Tilde{\delta}))$ is continuous with respect to its topology. So we only need
to show that any nonzero element of $L((T^{-1};\Tilde{\delta}))$ has a multiplicative inverse and the inverse map is continuous. 

Assume $0\ne z\in L((T^{-1};\Tilde{\delta}))$. One has $z=\sum_{i\le n}a_iT^i$. Then \[z=a_nT^n(1+w), w=\sum_{k\ge 1} w_k, w_k=T^{-n}a_{n}^{-1}a_{n-k}T^{n-k}.\]
Note that $deg(w_k)= -k$ for $k\ge1$, so $\sum_{k\ge 1} w_k$ is convergent, and $deg(w)\le -1$. Thus \[z^{-1}=(1+w)^{-1}T^{-n}a_n^{-1},\  (1+w)^{-1}=\sum_{j\ge0}(-1)^jw^j.\] So any nonzero element of $L((T^{-1};\Tilde{\delta}))$ has a multiplicative inverse and $L((T^{-1};\Tilde{\delta}))$ is a division ring. It follows from the above formula that  the inverse map is continuous.

\ep

We will
call $L((T^{-1};\Tilde{\delta}))$ a deformed Laurent series (division) ring with the standard degree map. 	
	Note that $L((T^{-1};\Tilde{\delta}))$ should not be confused with the deformation of algebras in \cite{f}.
	
	The purpose of this section is to find the necessary and sufficient conditions for $\Tilde{\delta}$ to make  $L((T^{-1};\Tilde{\delta}))$  a deformed Laurent series ring with the standard degree map. 
	
	As the multiplication of $L((T^{-1};\Tilde{\delta}))$ is  distributive over addition, one  need to know the formula for $T^i\cdot a$ for $a\in L, i\in \mathbb Z$, and $aT^i\cdot bT^j, i,j\in \mathbb Z$. One first consider $T\cdot a$.
	\ble 
	Assume that $L((T^{-1};\Tilde{\delta}))$ is a deformed Laurent series ring with the standard degree map. Then there exist maps $\delta_i:L\rightarrow L, i\in \mathbb Z_{\le1}$ such that \be\label{e23} T\cdot a=\delta_1(a)T+\sum_{i=0}^{-\infty}\delta_i(a)T^i, a\in L. \ee
	One has $\delta_i$ preserve addition for all $i$; and $\delta_1:L\rightarrow L$ is an injective ring homomorphism.
	\ele
	
	\bp
	by $ T(a+b)=Ta+Tb$, one has $\delta_i$ preserve addition for all $i$.
	$T1=T$ implies $\delta_1(1)=1$;
	$(Ta)b=Tab$ implies $\delta_1(ab)=\delta_1(a)\delta_1(b)$ for all $a,b\in L$, so $\delta_1:L\rightarrow L$ is a ring homomorphism. 
	
	Then $\text{Ker}(\delta_1)$
	is a proper ideal of $L$, which must be 0, thus $\delta_1$ is injective.
	
	\ep
%	It is easy to see that the multiplication is  continuous with respect to the topology induced by the degree map.
	
%	The following 2 examples show that $L((T^{-1};\Tilde{\delta}))$ generalizes some well known rings, including  the skew Laurent series ring and  the formal pseudo-differential operator ring.

	\ble\label{e68}
	Let $L((T^{-1};\Tilde{\delta}))$ be a deformed Laurent series ring with the standard degree map satisfying (\ref{e23}). Assume that $\delta_1=1$. Then \be \label{e45}T\cdot a=aT+\sum_{i=0}^{-\infty}\delta_i(a)T^i , a\in L,\ee
	and for $n,m\in \mathbb Z$, \be\label{e21} T^n\cdot a=\sum_{k\ge 0}\sum_{(j_1,\cdots,j_k )}\binom{n}{k}\delta_{j_1}\cdots \delta_{j_k}(a)T^{j_1+\cdots+j_k-k+n}\ee
	\[=a T^n+\sum_{k>0}\sum_{(j_1,\cdots,j_k )}\binom{n}{k}\delta_{j_1}\cdots \delta_{j_k}(a)T^{j_1+\cdots+j_k-k+n},\]
	\be\label{e16} aT^m\cdot bT^n=ab T^{m+n}+\sum_{k>0}\sum_{(j_1,\cdots,j_k )}\binom{m}{k}a\delta_{j_1}\cdots \delta_{j_k}(b)T^{j_1+\cdots+j_k-k+m+n}.\ee
	\ele
	\bp
	First we prove (\ref{e21}) in the case $n\ge 1$, by induction on $n$.
	
	The case $n=1$ is clear. Assume $n>1$ and (\ref{e21}) holds for all positive integers $ \le n$. Then 
	\begin{align*}
		T^{n+1}\cdot a &=T(T^n \cdot a)=T \cdot \left(\sum_{k \ge0}^{} \sum_{(j_1,\cdots,j_k)}^{}\binom{n}{k}\delta_{j_1} \cdots \delta _{j_k}(a)T^{j_1+\dots +j_k-k+n} \right)\\
		&=\sum_{k \ge0}^{} \sum_{(j_1,\dots,j_k)}^{}\binom{n}{k}(T \cdot \delta_{j_1} \dots \delta_{j_k}(a))T^{j_1+\cdots +j_k-k+n}\\
		&=\sum_{k \ge0}^{} \sum_{(j_1,\dots,j_k)}^{}\binom{n}{k}\left(\delta_{j_1}\cdots \delta_{j_k}(a)T+\sum_{l\le0}\delta _l(\delta_{j_1}\cdots \delta_{j_k}(a))T^l \right)\cdot T^{j_1+\cdots +j_k-k+n}\\
		&=\sum_{k \ge0}^{} \sum_{(j_1,\dots,j_k)}^{}\binom{n}{k}\delta_{j_1}\cdots \delta_{j_k}(a)T^{j_1+\cdots+ j_k-k+n+1}\\
		& +\sum_{k\ge 0}^{} \binom{n}{k} \sum_{(l,j_1,\cdots ,j_k)}^{} \delta_l \delta_{j_1}\cdots \delta _{j_k}(a)T^{j_1+\cdots+j_k+l-(k+1)+(n+1)} \\
		&=aT^{n+1}+\sum_{k>0}\sum_{(j_1,\cdots ,j_k)}[\binom{n}{k}+\binom{n}{k-1}] \delta_{j_1}\cdots \delta _{j_k}(a)T^{j_1+\cdots+j_k-k+(n+1)}\\
		&=aT^{n+1}+\sum_{k>0}\sum_{(j_1,\cdots ,j_k)}\binom{n+1}{k}\delta _{j_1}\cdots \delta _{j_k}(a)T^{j_1+\cdots+j_k-k+(n+1)}.
	\end{align*}
	Thus (\ref{e21}) holds for all $n\ge 1$.
	
	Next we consider the case $n=-1$. One has
	
	\begin{align*}
		&T\left(aT^{-1}+\sum_{k\ge1}\sum_{(j_1,\dots,j_k)}\binom{-1}{k}\delta_{j_1}\cdots \delta_{j_k}(a)T^{j_1+\cdots+j_k-k-1}\right)\\
		&=(aT+\sum_{i\le0}\delta_i(a)T^i)T^{-1}+\sum_{k\ge1 }\sum_{(j_1,\dots,j_k)}\binom{-1}{k}T\cdot \delta_{j_1}\cdots \delta_{j_k}(a)T^{j_1+\cdots+j_k-k-1}\\
		&=a+\sum_{i\le 0}\delta_i(a)T^{i-1}+\sum_{k\ge1 }\sum_{(j_1,\dots,j_k)}\binom{-1}{k}\left [ \delta_{j_1}\cdots\delta_{j_k}(a)\cdot T  +\sum_{l\le0}\delta_l(\delta_{j_1}\cdots \delta_{j_k}(a))T^l\right ] T^{j_1+\cdots +j_k-k-1}\\
		&=a+\sum_{i\le0}\delta_{i}(a)T^{i-1}+
		\sum_{k\ge1}\sum_{(j_1,\dots,j_k)}\binom{-1}{k}\delta_{j_1}\cdots \delta_{j_k}
		(a)T^{j_1+\cdots +j_k-k}\\
		& + \sum_{k\ge1}\sum_{(l,j_1,\dots,j_k)}\binom{-1}{k}\delta_{l}\delta_{j_1}\cdots \delta_{j_k}(a)T^{l+j_1+\cdots+j_k-k-1}\\
		&=a+\sum_{i\le0}[1+\binom{-1}{1}]\delta_{i}(a)T^{i-1}+
		\sum_{k\ge 2}\sum_{(j_1,\dots,j_k)}[\binom{-1}{k}+ \binom{-1}{k-1}]\delta_{j_1}\cdots \delta_{j_k}
		(a)T^{j_1+\cdots +j_k-k}\\
		& =a\\
	\end{align*}
	Multiply $T^{-1}$ on both sides from the left, one has  \begin{align*}
		T^{-1}a=aT^{-1}+\sum_{k\ge1}\sum_{(j_1,\dots,j_k)}\binom{-1}{k}\delta_{j_1}\cdots \delta_{j_k}(a)T^{j_1+\cdots+j_k-k-1}.
	\end{align*}
	So (\ref{e21}) holds for $n=-1$.
	Assume (\ref{e21}) holds for $n=-m$ with $m>0$, i.e.,
	\begin{align*}
		T^{-m}a=\sum_{k\ge0}\sum_{(l_1,\dots,l_k)}\binom{-m}{k}\delta_{l_1}\cdots\delta_{l_k}(a)T^{l_1+\cdots+l_k-k-m}.
	\end{align*}
	\begin{align*}
		T^{-1-m}a &=T^{-1}(T^{-m}a)=T^{-1}\sum_{k\ge0}\ \sum_{(l_1,\dots,l_k)}\binom{-m}{k} \delta_{l_1}\cdots \delta_{l_k}(a)T^{l_1+\cdots+l_k-k-m}\\
		&=\sum_{k\ge0}\ \sum_{(l_1,\dots,l_k)}\binom{-m}{k}\left [ T^{-1}\delta_{l_1}\cdots\delta_{l_k}(a) \right ]T^{l_1+\cdots +l_k-k-m}\\
		&=\sum_{k\ge0}\ \sum_{(l_1,\dots,l_k)}\binom{-m}{k}\left(\sum_{u\ge0}\ \sum_{(j_1,\dots,j_u)}\binom{-1}{u}\delta_{j_1}\cdots\delta_{j_u}(\delta_{l_1}\cdots \delta_{l_k}(a))T^{j_1+\cdots+j_u-u-1} \right)T^{l_1+\cdots+l_k-k-m}\\
		&=\sum_{k\ge0,u\ge0  }\ \sum_{(j_1,\cdots,j_u;l_1,\dots,l_k)}\binom{-m}{k}\binom{-1}{u}\delta_{j_1}\cdots \delta_{j_u}\delta_{l_1}\cdots\delta_{l_k}(a)T^{j_1+\cdots+j_u+l_1+\cdots+l_k-(u+1)-(m+k)}\\
		&=\sum_{p\ge0}\  \sum_{k,u\ge0}^{k+u=p}\ \sum_{(s_1,\dots,s_p)} \binom{-m}{k}\binom{-1}{u}\delta_{s_1}\cdots\delta_{s_p}(a)T^{s_1+\cdots+s_p-p-(m+1)}   \\
		&=\sum_{p\ge0}\ \sum_{(s_1,\dots,s_p)} \binom{-(m+1)}{p}\delta_{s_1}\cdots\delta_{s_p}(a)T^{s_1+\cdots +s_p-p-(m+1)}.
	\end{align*}
	So (\ref{e21}) holds for $-(m+1)$, thus it holds for all $n\in \mathbb Z_{<0}$. So (\ref{e21}) holds for all $n\in \mathbb Z$.
	
	It is clear that (\ref{e16}) follows from (\ref{e21}).
	\ep

	The lemma assumes $\delta_1=1$, if $1\ne \delta_1\in \text{Aut}(L)$, the corresponding results can be obtained similarly.
	
 If  $\delta_i= 0$ for all $i\le 0$, then $L((T^{-1};\Tilde{\delta}))$ is the usual field of  Laurent series.
 	Henceforth we will always assume that $\delta_1=1$ and $\delta_i\ne 0$  for some $i\le 0$, unless otherwise stated.
	
	Henceforth we will always assume that the characteristic of the field $L$ is 0.
	
	For a ring $A$ and $\emptyset\ne S,D\subseteq A$, let \[C(S,D)=\{b\in D|ab=ba, \forall a\in S\}\] be the centralizer of $S$ in $D$, and $C(A):=C(A,A)$ the center of $A$.
	
	\ble\label{e70}
	Let $L((T^{-1};\Tilde{\delta}))$ be a deformed Laurent series ring with the standard degree map satisfying (\ref{e45}), and  $\delta_i\ne 0$  for some $i\le 0$. Let $A=L((T^{-1};\Tilde{\delta}))$. Then $C(L,A)=L$; the center $C(A)$ is $\cap_{i\le 0} \text{Ker}(\delta_i)$, which is a field.

	\ele
	\bp Let $i_0=max\{i|i\le 0, \delta_i\ne 0\}. $
	Let   $a\in L\setminus \text{Ker}(\delta_{i_0})$. Then
	$C(a,A)\supseteq L$. 
	If $C(a,A)\ne L$, one can assume that there exists  $\sum_{i\le n}b_i T^i \in C(a,A) \setminus L$ with $n\ne 0$ and $b_n\ne 0$.
	\begin{align*}
		\sum_{i\le n}^{}b_i T^i \cdot a &=\sum_{i\le n} b_i (T^i \cdot a)\\
		&=\sum _{i\le n}b_i \left(a\cdot  T^i+\sum_{k>0} \sum_{(j_1,\cdots j_k)}\binom{i}{k}\delta _{j_1}\cdots \delta _{j_k}(a)T^{j_1+\cdots +j_k -k+i}\right)\\
		&=\sum_{i\le n}(b_i a)T^i+\binom{n}{1}b_n\delta_{i_0}(a)T^{i_0-1+n}+ \quad terms\ of\ degree<i_0-1+n \\
		a\sum_{i\le n }b_i T^i&=\sum_{i\le n }(ab_i) T^i.
	\end{align*}
	As $	\sum_{i\le n}^{}b_i T^i \cdot a=a\cdot \sum_{i\le n}^{}b_i T^i$ and $b_n\ne 0$, $\binom{n}{1}\delta_{i_0}(a)=0$. As the characteristic of $L$ is 0, $\delta_{i_0}(a)=0 $, which is a contradiction. Thus
	$C(a,A)\subseteq L$, one has $C(a,A)=L$ as $L \subseteq C(a,A)$.
	
	By $T\cdot a=aT+\sum_{i\le 0}\delta_i(a)T^i$, \[C(T,L)=\bigcap_{i\le 0}\text{Ker}(\delta_i).\] So 
	$C(A)=\bigcap_{i\le 0}\text{Ker}(\delta_i)$.
	
	As $L((T^{-1};\Tilde{\delta}))$ is  a division ring, its center $C(A)$ must be a field.
	
	% (Next we show that $K=\cap_{i\le 0} Ker(\delta_i)$ is a field.  Assume $a,b\in K$. It is clear that $a+b\in K$. We will show that $ab\in Ker\ \delta_{-i}, a^{-1}\in Ker\ \delta_{-i}$ (if $a\ne 0$) for  all $i\ge 0$ by induction. It holds for $i=0$ by direct verification. It is also easy to check that if it holds for $i$ then  it holds for $i+1$. Thus $ab, a^{-1}\in K$, and $K$ is a field.)
	\ep
		Let
	$K=\cap_{i\le 0} \text{Ker}(\delta_i)$. Then $A$ is a $K$-algebra.
	
		\ble\label{f26}	Let $L((T^{-1};\Tilde{\delta}))$ be a deformed Laurent series ring with the standard degree map satisfying (\ref{e45}).
	Then for $i\le 0$, \be\label{e47} \delta_i(ab)= a\delta_i(b)+\delta_i(a)b+ \sum_{k> 0}\sum_{ (l,j_1,\cdots,j_k)
	}^{j_1+\cdots+j_k-k+l=i}\binom{l}{k}\delta_l(a)\delta_{j_1}\cdots \delta_{j_k}(b), a,b\in L. \ee (See (\ref{e43}) for the meaning of 
	$\sum_{ (l,j_1,\cdots,j_k)}^{j_1+\cdots+j_k-k+l=i}$.)
	
		Let  $i_0=max\{i\in \mathbb Z_{\le 0}| \delta_i\ne 0\}.$ Then $\delta_{i_0}$ is a $K$-linear derivation. 
	Each $\delta_i, i\le0,$ is a $K$-linear differential operator of $L$.

	\ele
	
	\bp
	One has 
	\[T\cdot(ab)=ab T +\sum_{i\le0}^{}\delta_i(ab)T^i, a,b\in L,\] and
	\begin{align*}(Ta)b=& (aT+\sum_{i\le 0}\delta_i(a)T^i)b\\
		=&a(Tb)+\sum_{i\le 0}\delta_i(a)(T^i b)\\
		=&a(bT+\sum_{i\le 0}\delta_i(b)T^i)+\sum_{i\le 0}\delta_i(a)\left [ bT^i+\sum_{k\ge1}\sum_{(j_1,\cdots,j_k)}\binom{i}{k}\delta_{j_1}\cdots \delta_{j_k}(b)T^{j_1+\cdots +j_k-k+i  }\right ]\\
		=&abT+\sum_{i\le 0} \left( a\delta_i(b)T^i+\delta_i(a)bT^i \right) +
		\sum_{l\le 0}\sum_{k\ge 1 }\sum _{(j_1,\cdots,j_k)}\binom{l}{k}\delta_l(a)\delta_{j_1}\cdots \delta_{j_k}(b)T^{j_1+\cdots+j_k-k+l}
	\end{align*}
	by $T(ab)=(Ta)b$, compare the coefficient of each $T^i$, one has
	\begin{align*}
		\delta_{i}(ab)=a\delta_i(b)+\delta_i(a)b+\sum_{k\ge 1}\sum_{(l,j_1,\cdots,j_k)}^{l+j_1+\cdots+j_k-k=i}\binom{l}{k}\delta_l(a)\delta_{j_1}\cdots \delta_{j_k}(b), a,b\in L, i\le0.
	\end{align*}
	It follows that for $a\in K, b\in L$, 
	$\delta_i(ab)=a\delta_i(b)$. Thus each $\delta_i$ is $K$-linear.
	
	 It is clear that
	$\delta_{i_0}$ is a $K$-linear derivation. By induction one can easily show that for those $i$ with $\delta_{i}\ne0$, $\delta_{i}$ is a (high order) $K$-linear differential operator on $L$.
	
	\ep
	
	As the characteristic of $L$ is 0, the characteristic of its center $K$ is also 0.
	
%	The field $L$ is a $K$-algebra.
	
	All the tensor products below will be over $K$. 
	
	Let $B=End_{K-lin}(L)$ be the $K$-algebra of $K$-linear operators on $L$. Then  the tensor product $B\otimes B$ is also a $K$-algebra, which can be canonically identified with a subalgebra of $End_{K-lin}(L\otimes L)$.
	
	Let  \[\bold m:L\otimes L\rt L, a\otimes b\mapsto ab\] be the multiplication map. Then $\bold m$ is a homomorphism of rings, whose kernel is the ideal $I$ of $L\otimes L$ generated by all such elements $a\otimes 1-1\otimes a, a\in L$.
	
	 Let \[\Gamma=\{(\phi,\tau)\in (B\otimes B)\times B| \bold m\cdot \phi=\tau\cdot \bold m  \},\]
	i.e. $(\phi,\tau)\in \Gamma$ iff $(\phi,\tau)$ make the following diagram commutative.
	\begin{center}
		\begin{tikzcd}
			& L\otimes L \arrow[r,  "\bold m "]  \arrow[d, "\phi"] & L\arrow[d, "\tau"] \\
			& L\otimes L \arrow[r, "\bold m "] &L
		\end{tikzcd}
	\end{center}
	Let 	  \[C=\{\psi\in B\otimes B| \psi(I)\subseteq I \}.\] Then $C$ is a subalgebra of $B\otimes B$. It is clear that $(\phi,\tau)\in \Gamma$ iff $\phi\in C$. For  $\phi\in C$, it is clear that there exists a unique $\tau$ such that $(\phi,\tau)\in \Gamma$, and denote the $\tau$ by $\overline{\phi}$.
	
	The map $$\Psi:C\rt B, \phi\mapsto \overline{\phi},$$ is a ring homomorphism. It is clear that \[\text{Ker}(\Psi)=\{\phi\in B\otimes B| \phi(L\otimes L)\subseteq I \}.\] 
	Let $B_0=\Psi(C)$, which is a subalgebra of $B$.
The homomorphism $\Psi$ induces a ring isomorphism $C/\text{Ker}(\Psi)\rt B_0.$

	Let \[\Delta:B_0\rt C/\text{Ker}(\Psi)\] be the its inverse. Then
	for $\tau\in B_0$, $\Delta(\tau)=\phi+\text{Ker}(\Psi)$ iff
	$(\phi,\tau)\in \Gamma$.% If $\Delta(\tau)=\phi$ and $\Delta(\tau)=\phi_1$, then $\phi-\phi_1\in Ker(\zeta)$. So $\Delta:\Gamma_2\rt\Gamma_1$ can be regarded as a "multi-valued map" inverse to $\zeta$. One has \[ \Delta(\tau_1\tau_2)=\Delta(\tau_1)\Delta(\tau_2), \Delta(\tau_1+\tau_2)=\Delta(\tau_1)+\Delta(\tau_2).\]

%	\bex  $L=K(t), r\ge 0, $
%	$D^{(r)}(f)=\frac{f^{(r)}}{r!}$, the Hasse derivative, then \[   \Delta D^{(r)}(fg)=\sum_{i=0}^r D^{(i)}(f)\otimes D^{(r-i)}(g), f,g\in L; D^{(r)}\in \Gamma_2 \]
%	\eex

By (\ref{f26}), one has for $i\le0$, \be\label{e13}\begin{aligned}  \Delta\delta_i&=1\otimes\delta_i+\sum_{k>0}[ \sum_{(l,j_1,\cdots,j_k)}^{j_1+\cdots+j_k+l=i+k}\binom{l}{k}\delta_l\otimes\delta_{j_1}\cdots \delta_{j_k}]+\delta_i\otimes 1+\text{Ker}\ \Psi,
\end{aligned}\ee

or, \begin{align}\label{e31}
	\Delta\delta_i=1\otimes\delta_i+\sum_{k\ge0}[ \sum_{(l,j_1,\cdots,j_k)}^{j_1+\cdots+j_k+l=i+k}\binom{l}{k}\delta_l\otimes\delta_{j_1}\cdots \delta_{j_k}]+\text{Ker}\ \Psi,
\end{align}
	where $\delta_i\otimes 1$ in (\ref{e13}) is combined into the summation (regarded as $k=0$). 
Another equivalent formulation of (\ref{e13}) is
\be\label{e34}\Delta\delta_i=1\otimes\delta_i+\sum_{i<l\le 0}\delta_l\otimes\sum_{k=1}^{l-i}\binom{l}{k}[ \sum_{(j_1,\cdots,j_k)}^{j_1+\cdots+j_k=i+k-l}\delta_{j_1}\cdots \delta_{j_k}]+\delta_i\otimes 1+\text{Ker}\ \Psi.
\ee Note that $1\otimes\delta_i+\sum_{i<l\le 0}\delta_l\otimes\sum_{k=1}^{l-i}\binom{l}{k}[ \sum_{(j_1,\cdots,j_k)}^{j_1+\cdots+j_k=i+k-l}\delta_{j_1}\cdots \delta_{j_k}]+\delta_i\otimes 1$ is always a finite sum.

%So  \begin{align*}
%	(\Delta\delta_i)(a\otimes b)=\left(1\otimes\delta_i+\left(\sum_{k>0}[ \sum_{(l,j_1,\cdots,j_k)}^{j_1+\cdots+j_k+l=i+k}\binom{l}{k}\delta_l\otimes\delta_{j_1}\cdots \delta_{j_k}]+\delta_i\otimes 1\right)\right)(a\otimes b)
%\end{align*} and (\ref{e13}) holds. 
In particular $\delta_i\in B_0 $ for all $i\le1$. 
	\begin{rem}
		For $i=0,1,2,3$, one has
		\begin{align*}
			&\Delta \delta_0=1\otimes \delta_0+\delta_0\otimes 1+\text{Ker}\ \Psi,\\
			&\Delta \delta _{-1}=1\otimes \delta_{-1}+\delta_{-1}\otimes 1+\binom{0}{1}\delta_0\otimes \delta_0=1\otimes \delta_{-1}+\delta_{-1}\otimes 1+\text{Ker}\ \Psi,\\
			&\Delta \delta_{-2}=1\otimes \delta_{-2}+\delta_{-2}\otimes 1+\binom{-1}{1}\delta_{-1}\otimes \delta_0+\text{Ker}\ \Psi,\\
			&\Delta \delta_{-3}=1\otimes \delta_{-3}+\delta_{-3}\otimes 1+\binom{-1}{1}\delta_{-1}\otimes \delta_{-1}+\binom{-2}{1}\delta_{-2}\otimes \delta_{0}+\binom{-1}{2}\delta_{-1}\otimes \delta_{0}^2+\text{Ker}\ \Psi.
		\end{align*}
		
	\end{rem}

		We will use the following notation in later proof. According to the expression (\ref{e13}), we write \[\delta_l\otimes\delta_{j_1}\cdots \delta_{j_k}\in \Delta\delta_i, if\  j_1+\cdots+j_k+l=i+k.\]

	Next we will show that if (\ref{e13}) holds then (\ref{e16}) defines an associative product on $L((T^{-1}))$.
	
By  (\ref{e16}) and (\ref{e13}), for $a,b,c\in L\setminus \{0\}, m,n,l\in \mathbb Z$,
	one has
	\[ (aT^m\cdot bT^n)\cdot cT^l=\sum_{u\ge 0,k\ge 0}\sum_{(j_1,\cdots,j_k;s_1,\cdots,s_u )}\binom{m}{k}\binom{\sum j_i-k+m+n}{u}a\delta_{j_1}\cdots \delta_{j_k}(b)\delta_{s_1}\cdots \delta_{s_u}(c) T^{m+n+l+\sum j_i+\sum s_v-(k+u)} ,\] 
	where $\sum j_i$ means $\sum_{i=1}^k j_i$, which equals 0 if $k=0$. Similarly $\sum s_v=\sum_{v=1}^u s_v$,  which equals 0 if $u=0$.  Such conventions will be used henceforth. 
	
	Similarly one has
	\[ aT^m(bT^n\cdot cT^l)=\sum_{u\ge 0,k\ge 0}\sum_{(j_1,\cdots,j_k;s_1,\cdots,s_u )}\binom{m}{k}\binom{n}{u}a\delta_{j_1}\cdots \delta_{j_k}(b\delta_{s_1}\cdots \delta_{s_u}(c) )T^{m+n+l+\sum j_i+\sum s_v-(k+u)}, \]

	So, $(aT^m\cdot bT^n)\cdot cT^l=aT^m(bT^n\cdot cT^l)$ if for any fixed $p\in \mathbb Z_{\le 0}$, 
	
	\be\label{e29}\sum\binom{m}{k}\binom{\sum j_i-k+m+n}{t}\delta_{j_1}\cdots \delta_{j_k}(b)\delta_{s_1}\cdots \delta_{s_t}(c)= \sum\binom{m}{u}\binom{n}{r}\delta_{l_1}\cdots \delta_{l_u}\left(b\delta_{q_1}\cdots \delta_{q_r}(c)\right),\ee
	where the summation on the left (resp. on the right) is over $\Omega_1(p)$ (resp.$\Omega_2(p)$), where \[\Omega_1(p)=\{(j_1,\cdots,j_k;s_1,\cdots,s_t )|k, t\ge0; j_i, s_v\le0; \sum j_i+\sum s_v-(k+t)=p\} ,\]	
	 \[\Omega_2(p)=\{(l_1,\cdots,l_u;q_1,\cdots,q_r )|u,r\ge0; l_i, q_v\le0; \sum l_i+\sum q_v-(u+r)=p\}.\]
	
	The	left hand side  of (\ref{e29}) is \[\sum\binom{m}{k}\binom{\sum j_i-k+m+n}{t}\bold m\cdot(\delta_{j_1}\cdots \delta_{j_k}\otimes\delta_{s_1}\cdots \delta_{s_t})(b\otimes c),\] 
	
	The	right hand side  of (\ref{e29}) is \[\sum\binom{m}{u}\binom{n}{r}\bold m\cdot\Delta(\delta_{l_1}\cdots \delta_{l_u})\cdot (1\otimes\delta_{q_1}\cdots \delta_{q_r})(b\otimes c).\] 
		
	Recall that $B=End_{K-lin}(L)$.		 Let \be\label{e37} L(p)=\sum\binom{m}{k}\binom{\sum j_i-k+m+n}{t}\delta_{j_1}\cdots \delta_{j_k}\otimes \delta_{s_1}\cdots \delta_{s_t}\in B\otimes_K B,\ee where the summation is over $\Omega_1(p)$, and  \be\label{e32} R(p)=\sum\binom{m}{u}\binom{n}{r}\Delta(\delta_{l_1}\cdots \delta_{l_u})\cdot (1\otimes\delta_{q_1}\cdots \delta_{q_r})\in B\otimes_K B,\ee  where the summation is over $\Omega_2(p)$.
	Then $(aT^m\cdot bT^n)\cdot cT^l=aT^m(bT^n\cdot cT^l)$ {if} $L(p)=R(p)$ for all fixed $p\in \mathbb Z_{\le 0}$ and (\ref{e16}) holds.

	%that if \ref{e13} holds, then  $L(p)=R(p)$ for all fixed $p\in \mathbb Z_{\le 0}$.
	
	\ble \label{e66}	Assume there are additive maps $\delta_i:L\rightarrow L, i\in \mathbb Z_{\le 0},$ which satisfy
 (\ref{e13}). Let $ m,n,l\in \mathbb Z$. Then $L(p)=R(p)$ for all $p\in \mathbb Z_{\le 0}$. Thus if (\ref{e16}) holds, then
 $aT^m(bT^n\cdot cT^l)=(aT^m\cdot bT^n)\cdot cT^l$  for $a,b,c\in L$.
	\ele

\bp As the 2nd statement follows form the 1st one, one only needs to prove the 1st  statement.

Fix $p\in \mathbb Z_{\le0}$. 

	Note that if (\ref{e13}) holds, then $\Delta(\delta_{l_1}\cdots \delta_{l_u})\cdot (1\otimes\delta_{q_1}\cdots \delta_{q_r})=\Delta(\delta_{l_1})\cdots \Delta(\delta_{l_u})\cdot (1\otimes\delta_{q_1}\cdots \delta_{q_r})$ in $R(p)$ can further be expanded as the sum of terms of the form $\lambda\delta_{j_1}\cdots \delta_{j_k}\otimes \delta_{s_1}\cdots \delta_{s_t}$ with $(j_1,\cdots,j_k;s_1,\cdots,s_t )\in \Omega_1$ (and $\lambda\in \mathbb Z$), which is in $L(p)$.
 So we only need to check that the coefficient of any fixed term in $L(p)$ equals the coefficient of the same  term in the expansion of  $R(p)$. There are 3 steps.

1. Let $k,t\ge 0$ and $j_1,\cdots,j_k,s_1,\cdots,s_t\in \mathbb Z_{\le0}$. Let  \[ \phi_0({j_1},\cdots, {j_k};{s_1},\cdots, {s_t})\]

be the sum of coefficients of $\delta_{j_1}\cdots \delta_{j_k}\otimes\delta_{s_1}\cdots \delta_{s_t}$ in the expansion of all the terms in $R(p)$ of the form $\binom{m}{u}\binom{n}{0}\Delta(\delta_{l_1}\cdots \delta_{l_u})\cdot(1\otimes 1)$ (the terms with $r=0$ in (\ref{e32})).

%($\delta_{j_1}\cdots \delta_{j_k}\otimes\delta_{s_1}\cdots \delta_{s_t}$在所有右边这种形式的项$\binom{m}{i}\binom{n}{0}\Delta(\delta_{l_1}\cdots \delta_{l_i})\cdot(1\otimes 1)$ 的展开式中出现的系数之和)

For example, if $t=0$ then $\phi_0({j_1},\cdots ,{j_k};\emptyset)=\binom{m}{k}$, if $k=0$ then $\phi_0(\emptyset; {s_1},\cdots ,{s_t})=\binom{m}{t}$.\\[2mm]

First we will show
\begin{align}\label{e30}
	\phi _0(j_1,\cdots ,j_k;s_1,\cdots,s_t)=\sum_{i=0}^{t} \binom{m}{k+i}\binom{k+i}{k}\binom{j_1+\cdots+j_k}{t-i}.
\end{align}
To avoid complicated subscripts, we illustrate it in the case $k=2,t=3$ below. The general case can be treated similarly.

We will compute the coefficient of $\delta_{j_1}\delta_{j_2}\otimes\delta_{s_1}\delta_{s_2}\delta_{s_3}$ in the expansion of  terms of the type $\binom{m}{u}\binom{n}{0}\Delta(\delta_{l_1}\cdots \delta_{l_u})\cdot (1\otimes 1)$, which equals $\binom{m}{u}\Delta(\delta_{l_1})\cdots \Delta(\delta_{l_u})$. One must have $2\le u\le 2+3$. We will compute the coefficient of $\delta_{j_1}\delta_{j_2}\otimes\delta_{s_1}\delta_{s_2}\delta_{s_3}$ in the expansion of those terms of the type $\Delta(\delta_{l_1})\cdots \Delta(\delta_{l_{2+i}})$ with $i=0,1,2,3$.

Assume $i=0$. By (\ref{e31}) one has
\begin{align*}
	\Delta\delta_{l_1}=\sum_{k\ge 0} \sum_{(u,p_1,\cdots,p_k)}^{p_1+\cdots+p_k+u=l_1+k}\binom{u}{k}\delta_u\otimes\delta_{p_1}\cdots \delta_{p_k}+1\otimes\delta_{l_1}+\text{Ker}\ \Psi, 
\end{align*}

\begin{align*}
	\Delta\delta_{l_2}=\sum_{r\ge 0} \sum_{(v,q_1,\cdots,q_r)}^{q_1+\cdots+q_r+v=l_2+r}\binom{v}{r}\delta_v\otimes\delta_{q_1}\cdots \delta_{q_r}+1\otimes\delta_{l_2}+\text{Ker}\ \Psi.
\end{align*}

If $\delta_{j_1}\delta_{j_2}\otimes\delta_{s_1}\delta_{s_2}\delta_{s_3}$ appears in the expansion of $\Delta\delta_{l_1}\cdot\Delta\delta_{l_2}$, then for some $u=0,1,2,3$, $\delta_{j_1}\otimes\delta_{s_1}\cdots\delta_{s_u}$ is in the expansion of $\Delta\delta_{l_1}$, and $\delta_{j_2}\otimes\delta_{s_{u+1}}\cdots\delta_{s_3}$ is in the expansion of $\Delta\delta_{l_2}$. Thus the coefficient of  $\delta_{j_1}\delta_{j_2}\otimes\delta_{s_1}\delta_{s_2}\delta_{s_3}$ in those $\binom{m}{2}\Delta\delta_{l_1}\cdot\Delta\delta_{l_2}$ is \be\label{e33}\binom{m}{2}\sum_{u=0}^3\binom{j_1}{u}\binom{j_2}{3-u}=\binom{m}{2} \binom{j_1+j_2}{3}.\ee 

Note that (\ref{e33}) is just the  summand of the right side of (\ref{e30}) with $i=0$ in the case $k=2,t=3$.

Assume $i=1$. Let $\alpha=\delta_{j_1}\delta_{j_2}\otimes\delta_{s_1}\delta_{s_2}\delta_{s_3}$. If $\alpha$ appears in the expansion of $\Delta\delta_{l_1}\cdot\Delta\delta_{l_2}\cdot\Delta\delta_{l_3}$, then there are 3 cases:

(a) $\alpha=(\delta_{j_1}\otimes \delta_{s_1}\cdots \delta_{s_{u}})\cdot (\delta_{j_2}\otimes \delta_{s_{u+1}}\cdots \delta_{s_2})\cdot (1\otimes\delta_{s_3}), u=0,1,2$, with \[\delta_{j_1}\otimes \delta_{s_1}\cdots \delta_{s_{u}}\in\Delta\delta_{l_1},\delta_{j_2}\otimes \delta_{s_{u+1}}\cdots \delta_{s_2}\in\Delta\delta_{l_2}, 1\otimes\delta_{s_3}\in\Delta\delta_{l_3};\] 

(b) $\alpha=(\delta_{j_1}\otimes \delta_{s_1}\cdots \delta_{s_{u}})\cdot (1 \otimes \delta_{s_{u+1}})\cdot (\delta_{j_2}\otimes \delta_{s_{u+2}}\cdots \delta_{s_3}),\ u=0,1,2$, with \[\delta_{j_1}\otimes \delta_{s_1}\cdots \delta_{s_{u}}\in\Delta\delta_{l_1}, 1\otimes\delta_{s_{u+1}}\in\Delta\delta_{l_2}, \delta_{j_2}\otimes \delta_{s_{u+2}}\cdots \delta_{s_3}\in\Delta\delta_{l_3};\] 

(c)  $\alpha=(1\otimes \delta_{s_1})\cdot (\delta_{j_1}\otimes \delta_{s_2}\cdots \delta_{s_{u}})\cdot (\delta_{j_2}\otimes \delta_{s_{u+1}}\cdots \delta_{s_3}), u=1,2,3$, with \[ 1\otimes\delta_{s_1}\in\Delta\delta_{l_1},\delta_{j_1}\otimes \delta_{s_2}\cdots \delta_{s_{u}}\in\Delta\delta_{l_2},\delta_{j_2}\otimes \delta_{s_{u+1}}\cdots \delta_{s_3}\in\Delta\delta_{l_3}.\]

Thus the sum of coefficient of  $\delta_{j_1}\delta_{j_2}\otimes\delta_{s_1}\delta_{s_2}\delta_{s_3}$ in  $\binom{m}{3}\Delta\delta_{l_1}\Delta\delta_{l_2}\Delta\delta_{l_3}$ is \[\binom{m}{3}\binom{3}{2}\sum_{u=0}^2\binom{j_1}{u}\binom{j_2}{3-1-u}=\binom{m}{2+1}\binom{2+1}{2} \binom{j_1+j_2}{3-1},\] which is just the  summand of the right side of (\ref{e30}) with $i=1$  in the case $k=2,t=3$. 

If $i=2,3$, by similar computations, one finds that the sum of coefficient of  $\delta_{j_1}\delta_{j_2}\otimes\delta_{s_1}\delta_{s_2}\delta_{s_3}$ in $\binom{m}{2+i}\Delta\delta_{l_1}\cdots\Delta\delta_{l_{2+i}}$ is $\binom{m}{2+i}\binom{2+i}{2}\binom{j_1+j_2}{3-i}.$ So (\ref{e30}) holds in the case $k=2,t=3$. 

In the general case,
$\delta_{j_1}\cdots \delta_{j_k}\otimes\delta_{s_1}\cdots \delta_{s_t} $ occurs in the expansion of  $\binom{m}{k+i}\Delta(\delta_{l_1})\cdots \Delta(\delta_{l_{k+i}})$ for $i=0,\cdots,t$. For some fixed $i=0,\cdots,t$,
there are $\binom{k+i}{k}$ cases to write $\delta_{j_1}\cdots \delta_{j_k}\otimes\delta_{s_1}\cdots \delta_{s_t}$ as a product of items of the form
$\delta_{j_u}\otimes\delta_{s_w}\cdots \delta_{s_{w+p}}$ ($u=1,\cdots,k$) and $1\otimes \delta_{s_q}$ (there are $i$ such  $1\otimes \delta_{s_q}$ ). In each case, there are also sub-cases as in above example, and the sum of coefficient of $\delta_{j_1}\cdots \delta_{j_k}\otimes\delta_{s_1}\cdots \delta_{s_t} $ in all the sub-cases (subordinate to a given case) equals
\[\sum_{w_1,\cdots,w_k\ge 0}^{w_1+\cdots w_k=t-i} \binom{j_1}{w_1}\cdots\binom{j_k}{w_k}=\binom{j_1+\cdots+j_k}{t-i}.\]

So for a fixed $i$, the sum of coefficient of $\delta_{j_1}\cdots \delta_{j_k}\otimes\delta_{s_1}\cdots \delta_{s_t} $ in the expansion of $\binom{m}{k+i}\Delta(\delta_{l_1})\cdots \Delta(\delta_{l_{k+i}})$
is $\binom{m}{k+i}\binom{k+i}{k}\binom{j_1+\cdots+j_k}{t-i}$,
thus (\ref{e30}) holds in the general case.

2. Let \[\phi({j_1},\cdots, {j_k};{s_1},\cdots, {s_t}),\ \ (k,t\ge 0)\]
be the sum of coefficient of $\delta_{j_1}\cdots \delta_{j_k}\otimes\delta_{s_1}\cdots \delta_{s_t}$ in the expansion of all the terms in $R(p)$ of the form $\binom{m}{u}\binom{n}{r}\Delta(\delta_{l_1}\cdots \delta_{l_u})\cdot (1\otimes\delta_{q_1}\cdots \delta_{q_r}) $ 
where  $ u\ge 0, r\ge 0, {(l_1,\cdots,l_u;q_1,\cdots,q_r )}$ satisfying $\sum l_i+\sum q_i-(u+r)=p$. see (\ref{e32}).

For example, if $k=0$ then $\phi(\emptyset; {s_1},\cdots, {s_t})=\binom{m+n}{t}$.

Now we illustrate the equality
\be\label{e35}\phi({j_1},\cdots, {j_k};{s_1},\cdots, {s_t})
= \sum_{i=0}^t \binom{n}{i}\phi_0({j_1},\cdots, {j_k};{s_1},\cdots, {s_{t-i}})\ee
in the case $t=2$:

$\delta_{j_1}\cdots \delta_{j_k}\otimes\delta_{s_1} \delta_{s_2}$ occurs in
$\binom{m}{u}\binom{n}{0}\Delta(\delta_{l_1}\cdots \delta_{l_u})\cdot (1\otimes1)$ with coefficient $ \binom{n}{0}\phi_0(j_1,\cdots ,j_k;s_1,s_2)$;

$\delta_{j_1}\cdots \delta_{j_k}\otimes\delta_{s_1} \delta_{s_2}=\delta_{j_1}\cdots \delta_{j_k}\otimes\delta_{s_1}\cdot 1\otimes \delta_{s_2}$ occurs in
$\binom{m}{u}\binom{n}{1}\Delta(\delta_{l_1}\cdots \delta_{l_u})\cdot (1\otimes \delta_{s_2})$ with coefficient $ \binom{n}{1}\phi_0(j_1,\cdots, j_k;s_1) $;

$\delta_{j_1}\cdots \delta_{j_k}\otimes\delta_{s_1} \delta_{s_2}=\delta_{j_1}\cdots \delta_{j_k}\otimes 1\cdot 1\otimes\delta_{s_1}\delta_{s_2}$ occurs in
$\binom{m}{u}\binom{n}{2}\Delta(\delta_{l_1}\cdots \delta_{l_u})\cdot (1\otimes \delta_{s_1}\delta_{s_2})$ with coefficient $ \binom{n}{2}\phi_0(j_1,\cdots, j_k;\emptyset)$.

These exhaust all the possibilities for  $\delta_{j_1}\cdots \delta_{j_k}\otimes\delta_{s_1} \delta_{s_2}$ to occur in the expansion of the terms in $R(p)$ of the form $\binom{m}{u}\binom{n}{r}\Delta(\delta_{l_1}\cdots \delta_{l_u})\cdot (1\otimes\delta_{q_1}\cdots \delta_{q_r}) $, thus (\ref{e35}) holds in the case $t=2$. It is easy to see that (\ref{e35}) holds in the general case .

3. Now we compute $\phi({j_1},\cdots, {j_k};{s_1},\cdots, {s_t})$.

Recall that 
\be\label{e27} \binom{u}{v+w}\binom{v+w}{w}=\binom{u}{v}\binom{u-v}{w}, v,w\in \mathbb Z_{\ge0}, u\in \mathbb Z. \ee
Let $J=j_1+\cdots+j_k$. Then
\be\begin{aligned}
	\phi({j_1},\cdots, {j_k};{s_1},\cdots, {s_t})&=\sum_{q=0}^t \phi_0({j_1},\cdots, {j_k};{s_1},\cdots, {s_{t-q}})\cdot \binom{n}{q}\\
	&=\sum_{q=0}^t[ \sum_{i=0}^{t-q} \binom{m}{k+i}\binom{k+i}{k}\binom{J}{t-q-i}] \binom{n}{q}\\
	&=\sum_{q=0}^t[ \sum_{i=0}^{t-q} \binom{m}{k}\binom{m-k}{i}\binom{J}{t-q-i}] \binom{n}{q}\ \ \ (by\ (\ref{e27}))\\
	&=\binom{m}{k}\sum_{q=0}^t[ \sum_{i=0}^{t-q}\binom{m-k}{i}\binom{J}{t-q-i}] \binom{n}{q}\\
	&=\binom{m}{k}\sum_{q=0}^t\binom{m-k+J}{t-q}\binom{n}{q}\\
	&=\binom{m}{k}\binom{m+n-k+J}{t}.
\end{aligned}\ee
Thus one has \[\phi({j_1},\cdots ,{j_k};{s_1},\cdots, {s_t})=\binom{m}{k}\binom{m+n-k+ (j_1+\cdots j_k)}{t}\]
and $L(p)=R(p)$ (see (\ref{e37}) and (\ref{e32})).
\ep
\bco \label{e67}	Assume there are additive maps $\delta_i:L\rightarrow L, i\in \mathbb Z_{\le 0},$ which satisfy
(\ref{e13}). Then there exists some unique product  $\Tilde{\delta}$  on  $L((T^{-1}))$ satisfying (\ref{e16}) such that 	$L((T^{-1};\Tilde{\delta}))$ is a deformed Laurent series ring with the standard degree map.
\eco
\bp
It has been show in last lemma that if (\ref{e13}) and (\ref{e16}) holds then for any $a,b,c\in L, m,n,l\in \mathbb Z$,
one has
$(aT^m\cdot bT^n)\cdot cT^l=aT^m\cdot (bT^n\cdot cT^l)$. 

Extending the product 
(\ref{e16}) to $L((T^{-1}))$ by distributive law and continuity.
Then for $z_1,z_2,z_3\in L((T^{-1}))$ with $z_1=\sum_{i\le m} a_i T^i, z_2=\sum_{j\le n } b_j T^j, z_3=\sum_{k\le l} c_k T^k$, one has 
\[(z_1z_2)z_3=\sum_{i\le m,j\le n,k\le l}(a_i T^i\cdot  b_j T^j) c_k T^k=\sum_{i\le m,j\le n,k\le l}a_i T^i ( b_j T^j\cdot c_k T^k)=z_1(z_2z_3).\]
(One has $(z_1z_2)z_3=z_1(z_2z_3)=0$ if one of $z_1,z_2,z_3$ is 0.)

Then it is easy to see that $L((T^{-1}))$ with this product satisfies all the conditions in Definition \ref{e39}, thus is a deformed Laurent series ring with the standard degree map. The uniqueness of the product is clear.

\ep
To sum up, here is the main result of this section.
%\begin{theorem}(old)\label{f22} Let $L$ be a field of characteristic 0.	Let $L((T^{-1}))$  have  the standard degree map and the usual addition. Assume  $\Tilde{\delta}$ is a multiplication on  $L((T^{-1}))$ such that 
%	\be\label{f25} t\cdot a=at+\sum_{i=0}^{-\infty}\delta_i(a)T^i, a\in L,\ee where $\delta_i:L\rightarrow L, i\in \mathbb Z_{\le 0} ,$ are additive maps. 
	
%	(1) One has that	$L((T^{-1};\Tilde{\delta}))$ is a deformed Laurent series ring with the standard degree map $deg$, iff for all $i\le 0$,
%	\be\label{e41}  \delta_i(ab)= a\delta_i(b)+\delta_i(a)b+ \sum_{k> 0}\sum_{ (l,j_1,\cdots,j_k)	}^{j_1+\cdots+j_k-k+l=i}\binom{l}{k}\delta_l(a)\delta_{j_1}\cdots \delta_{j_k}(b)
%	.\ee
	
%	Assume (\ref{e41}) holds, then  one has
%	\[(\sum_{i\le m} a_i T^i)(\sum_{j\le n } b_j T^j)=\sum_{i\le m,j\le n} a_i T^i\cdot b_j T^j, where\]	
%	\be\label{f24} a_iT^i\cdot b_jT^j=a_ib_jT^{i+j}+\sum_{k\ge 1}\sum_{ (j_1,\cdots,j_k )}\binom{i}{k}a_i\cdot\delta_{j_1}\cdots \delta_{j_k}(b_j)T^{j_1+\cdots+j_k-k+i+j}, %a,b\in L;m,n\in \mathbb Z
%	\ee
	
%	(2)	$L((T^{-1};\Tilde{\delta}))$ is a topological  division ring  which is  complete with respect to the topology induced by the degree map $deg$, and the center of  $L((T^{-1};\Tilde{\delta}))$  is the field  $K=\cap_{i\le 0} Ker(\delta_i)$.
	
%\end{theorem}
\begin{theorem}\label{f22} Let $L$ be a field of characteristic 0.	Let $L((T^{-1}))$  have  the standard degree map and the usual addition.
	Assume there are additive maps $\delta_i:L\rightarrow L, i\in \mathbb Z_{\le 0}.$ 	
	Then the following are equivalent.
	
	(1) There exists some unique multiplication  $\Tilde{\delta}$  on  $L((T^{-1}))$ such that 	$L((T^{-1};\Tilde{\delta}))$ is a deformed Laurent series ring with the standard degree map, satisfying
	\be\label{f25} T\cdot a=aT+\sum_{i=0}^{-\infty}\delta_i(a)T^i, a\in L,\ee
	 
	 (2) For all $ i\in \mathbb Z_{\le 0} $,	\be\label{e41}  \delta_i(ab)= a\delta_i(b)+\delta_i(a)b+ \sum_{k> 0}\sum_{ (l,j_1,\cdots,j_k)
	}^{j_1+\cdots+j_k-k+l=i}\binom{l}{k}\delta_l(a)\delta_{j_1}\cdots \delta_{j_k}(b), a,b\in L
	.\ee

Assume (1) or (2) holds, then the formula for the multiplication  $\Tilde{\delta}$ is
	\[(\sum_{i\le m} a_i T^i)(\sum_{j\le n } b_j T^j)=\sum_{i\le m,j\le n} a_i T^i\cdot b_j T^j, where\]	
	\be\label{f24} a_iT^i\cdot b_jT^j=a_ib_jT^{i+j}+\sum_{k\ge 1}\sum_{ (j_1,\cdots,j_k )}\binom{i}{k}a_i\cdot\delta_{j_1}\cdots \delta_{j_k}(b_j)T^{j_1+\cdots+j_k-k+i+j}. %a,b\in L;m,n\in \mathbb Z
	\ee
	
Moreover,	$L((T^{-1};\Tilde{\delta}))$ is a topological  division ring  which is  complete with respect to the topology induced by the standard degree map, and the its center  is the field  $K=\cap_{i\le 0} \text{Ker}(\delta_i)$.
	
\end{theorem}
\bp
(1) implies (2) by Lemma \ref{f26}. (2) implies (1) by Lemma \ref{e66} and Corollary \ref{e67}. The formula (\ref{f24}) is proven in Lemma \ref{e68}.
\ep
%\bp
%What is left to be  shown is that $L((T^{-1};\Tilde{\delta}))$ is a division ring and the inverse map is continuous. Assume $0\ne z\in L((T^{-1};\Tilde{\delta}))$. One has $z=\sum_{i\le n}a_iT^i$. Then \[z=a_nT^n(1+w), w=\sum_{i\ge 1} w_i, w_i=T^{-n}a_{n}^{-1}a_{n-i}T^{n-i}.\]
%Note that $deg(w_i)= -i$ for $i\ge1$, so $\sum_{i\ge 1} w_i$ is convergent, and $deg(w)\le -1$. Thus \[z^{-1}=(1+w)^{-1}T^{-n}a_n^{-1}, (1+w)^{-1}=\sum_{j\ge0}(-1)^jw^j.\] So any nonzero element of $L((T^{-1};\Tilde{\delta}))$ has a multiplicative inverse and $L((T^{-1};\Tilde{\delta}))$ is a division ring. It is clear that  the inverse map is continuous.

%\ep
It is clear by  (\ref{f24}) that the multiplication $\Tilde{\delta}$ in $L((T^{-1};\Tilde{\delta}))$ is determined by the set of $\delta_i,i\le0$ in (\ref{f25}). So one can write $\Tilde{\delta}=(\delta_i)_{i\le0}.$

By (\ref{f24}), the associated graded algebra of $L((T^{-1};\Tilde{\delta}))$ with respect to the standard degree map $deg$ is just the usual field  $L((T^{-1}))$ of Laurent series.

The proof of the following result is direct and is omitted.
\ble
Let $L((T^{-1};\widetilde{\delta}))$ be a deformed Laurent series ring and $\widetilde{\delta}=(\delta_i)_{i\le0}$. Assume that $w:L\rt \mathbb R_{-\infty}$ is a degree map such that $w(\delta_i(a))\le w(a) $ for all $a\in L$ and all $i\le0$.  Then the set of elements $\sum_{i\le n} a_i T^i$ in  $L((T^{-1};\widetilde{\delta}))$, such that $w(a_i)$ is a bounded function of $i$, forms a subring of  $L((T^{-1};\widetilde{\delta}))$.
\ele
We will denote this ring by $L((T^{-1};\widetilde{\delta}))^b$.\\[1mm]

Now we convert Theorem (\ref{f22}) to right $L$-vector space of  Laurent seires.

Let	$L_\textbf{r}((T^{-1}))=\{\sum_{i\le k}T^ia_i|a_i\in L, k\in \mathbb Z\}$, the right $L$-vector space of  Laurent seires. One has the map
\be\label{f20}\phi:L((T^{-1})) \rt L_\textbf{r}((T^{-1})),\ \ \sum_{i\le k}a_iT^i \mapsto \sum_{i\le k}T^ia_i.\ee
The map
\be\label{f21}deg_r:L_\textbf{r}((T^{-1})) \rt \mathbb{Z_{-\infty}},\ \ \sum_{i\le k}T^ia_i \mapsto k\ (a_k\ne 0), 0\mapsto -\infty,\ee
is called the standard degree map on $L_\textbf{r}((T^{-1}))$. Note that $deg_r= deg\cdot\phi$.

Analogous to Def. \ref{e39}, if there exists a multiplication $\widetilde{\eta}$ on $L_\textbf{r}((T^{-1}))$ such that the corresponding axioms 1-4 with the obvious modification (for example, in Axiom 4, the degree map  $deg$ is replaced by  $deg_r$) are satisfied, then $L_\textbf{r}((T^{-1}))$ is called a right  deformed Laurent series ring (with the standard degree map), and denoted by $L_\textbf{r}((T^{-1};\widetilde{\eta}))$.

Assume that $L_\textbf{r}((T^{-1};\Tilde{\eta}))$ is a right deformed Laurent series ring. Then there exist maps $\eta_i:L\rightarrow L, i\in \mathbb Z_{\le1},$ such that \be a\cdot T=T\eta_1(a)+\sum_{i=0}^{-\infty}T^i\eta_i(a), a\in L. \notag\ee
Then $\eta_i$ preserve addition for all $i$, and $\eta_1:L\rightarrow L$ is an injective ring homomorphism.

We will always assume $\eta_1=1$. As $L_\textbf{r}((T^{-1};\Tilde{\eta}))$ is associative, one has $(ab)T=a(bT)$, from which one can get following equations (as in Lemma \ref{f26}):
	\be\label{f19}\eta_i(ab)= a\eta_i(b)+\eta_i(a)b+ \sum_{k> 0}\sum_{ (l,j_1,\cdots,j_k)
}^{j_1+\cdots+j_k-k+l=i}\binom{l}{k}\eta_{j_1}\cdots \eta_{j_k}(a)\eta_l(b), a,b\in L, i\le0.\ee
Now we have the following result analogous to Theorem \ref{f22}.
\begin{coro}\label{f30} Let $L$ be a field of characteristic 0. Let $L_\textbf{r}((T^{-1}))$  have  the standard degree map $deg_r$ and the usual addition.
	Assume there are additive maps $\eta_i:L\rightarrow L, i\in \mathbb Z_{\le 0}.$ 	
	Then the following are equivalent.
	
	(1) There exists some unique multiplication  $\Tilde{\eta}$ on  $L_\textbf{r}((T^{-1}))$ such that 
	\be\label{f23} a\cdot T=Ta+\sum_{i=0}^{-\infty}T^i\eta_i(a), a\in L. \ee
	
	(2)
	 (\ref{f19}) holds for all $i\le 0$ and $a,b\in L$.

	Assume (\ref{f19}) holds, then  one has
		\[(\sum_{i\le m}  T^ia_i)(\sum_{j\le n } T^jb_j )=\sum_{i\le m,j\le n} T^i a_i\cdot  T^jb_j, where\]	
\[	T^ia_i\cdot T^jb_j=T^{i+j}a_ib_j+\sum_{k>0}\sum_{(j_1,\cdots,j_k)}\binom{j}{k}T^{j_1+\cdots+j_k-k+i+j}\eta_{j_1}\cdots \eta_{j_k}(a_i)b_j.\]
	Moreover,
	$L_\textbf{r}((T^{-1};\Tilde{\eta}))$ is a topological division ring  which is  complete with respect to the topology induced by the degree map $deg_r$, and the center of  $L_\textbf{r}((T^{-1};\Tilde{\eta}))$  is the field  $K=\cap_{i\le 0} \text{Ker}(\eta_i)$.
	
\end{coro}

The sketch of the proof is as follows. Given $L_\textbf{r}((T^{-1};\Tilde{\eta}))$,  there exists some $L((T^{-1};\Tilde{\delta}))$ such that the map  $\phi:	L_\textbf{r}((T^{-1};\Tilde{\eta}))\rt L((T^{-1};\Tilde{\delta})),  \sum_{i\le k}T^i a_i\mapsto \sum_{i\le k}a_iT^i $ is an  anti-isomorphism of rings, i.e. $\phi$ is a bijection preserving addition and $\phi(f\cdot g)=\phi(g)\phi(f)$.

Applying $\phi$ on both sides of (\ref{f23}) one has \[Ta=aT+\sum_{i=0}^{-\infty}\eta_i(a)T^i, a\in L . \] So, $\delta_i=\eta_i$ for all $i\le0.$  The above corollary can be proven using this anti-isomorphism and Theorem \ref{f22}.

\section{Construction of $\widetilde{D}_{r,s}$  and $\check{D}_{r,s}$ and embedding of $\widehat{D}_{r,s}$}

 \setcounter{equation}{0}\setcounter{theorem}{0}
 
 Let $K$ be a field of characteristic 0.

Assume that $(r,s )\in \mathbb Z^2$ with $s \ne 0$, where $r,s$ are not required to be coprime. Let $\nu =r +s$. Assume $\nu>0$.

Let $\alpha_{r,s}, T_{r,s}$ be two indeterminants associated with $(r,s)$, and denote $\alpha=\alpha_{r,s}, T=T_{r,s}$ for simplicity.

For $k\ge 0$, let \[d_k=[1,\nu]_k/s^k=(\frac{\nu}{s})^k[\nu^{-1}]_k.\] In particular,
$d_0=1,d_1=1/s$.

Let $L=K(\alpha)$, and $\delta _i\in End_{K-lin}(L)$ be defined as follows.  
\begin{equation}\label{e64}
	\delta _i=
	\left\{
	\begin{array}{ll}
		\frac{d_k}{k!} (\frac{d}{d\alpha})^k, \ \ & i=1-k\nu \text{ with } k\ge 0;\\
		0, & i\le 1 \text{ and } i\not\equiv 1\ (mod\ \nu).
	\end{array}
	\right.
\end{equation}
Note that $\delta_1=1$ is the identity map, and the set of linear operators $\delta _{1-k\nu}, k=0,1,2,\cdots,$ are linearly independent, and commute with each other.
\par
\noindent
\begin{prop}   \label{fp1} 
	 The set of linear operators $\delta_i,i\in \mathbb Z_{\le0}$, satisfy (\ref{e41}), so there exists 
	 a multiplication  $\Tilde{\delta} $ on  $L((T^{-1}))$ such that
	\be\label{f33}T\cdot f=f\cdot T+\sum_{k=1}^{\infty } \delta _{1-k\nu }(f)\cdot T^{1-k\nu }
	=f\cdot T+s^{-1}f'T^{1-\nu}+\frac{1}{2!}\frac{1-\nu}{s^2}f^{(2)}T^{1-2\nu}+\cdots, f\in L,\ee and $L((T^{-1};\Tilde{\delta}))$ is a deformed Laurent series ring. The multiplication formula in this ring   is as follows:
	\[\sum_{i\le m} a_i T^i\cdot \sum_{j\le n}b_j T^j= \sum_{i\le m,j\le n} a_i T^i\cdot b_j T^j,\  where\] \begin{align*} &a_i T^i\cdot b_j T^j=a_i( T^i\cdot b_j )T^j=a_i(\sum_{k\ge 0}\lambda_{i}^k b_j^{(k)}T^{i-k\nu} )T^j=\sum_{k\ge 0}\lambda_{i}^k a_i b_j^{(k)}T^{i-k\nu+j}, and \end{align*}
	\be\label{f15} \lambda_{i}^k=(\frac{\nu}{s} )^k \sum_{l=1}^k(-1)^{k-l}\binom{i}{l}\binom{i-l-1}{k-l}\binom{l\nu^{-1}}{k}\ for\ k>0; \lambda_{i}^0=1.\ee
	In particular, if $(r,s)=(0,1)$ then $\lambda_{i}^k=\binom{i}{k}$.
\end{prop} 

Denote this ring by $\widetilde{D}_{r,s}$, and also by  $K(\alpha)((T^{-1};\Tilde{\delta}_{r,s}))$.

It follows from (\ref{f33}) and (\ref{e21}) that
	for $f\in K(\alpha)$ and $n\in \mathbb Z$, one has \be\label{g2}[T^n,f]=\binom{n}{1}\sum_{k>0}\frac{d_k}{k!}(\frac{d}{d\alpha})^k f\cdot  T^{n-k\nu}+ \binom{n}{2}\sum_{k_1,k_2>0}\frac{d_{k_1}d_{k_2}}{k_1!k_2!}(\frac{d}{d\alpha})^{k_1+k_2} f\cdot T^{n-(k_1+k_2)\nu}+ \cdots .   \ee In particular, 
	\be\label{g1}[T^n,\alpha]=ns^{-1}T^{n-\nu}.\ee

We need some preparation before we prove Proposition \ref{fp1}.
\ble\label{fl3}
One has for $a\in \mathbb{R}, m,t\ge 0$,
\begin{align}
	\sum_{j=0}^{m}(-1)^j\binom{a}{j}\binom{a-j}{t}=(-1)^{m}\binom{a}{t}\binom{a-t-1}{m}.  \label{f9}
\end{align}
\ele
\bp Assume that $a\in \mathbb Z_{\ge m+t}$. Then
\begin{align*}
	&\ \sum_{j=0}^{m}(-1)^j\binom{a}{j}\binom{a-j}{t}=\sum_{j=0}^{m}(-1)^j \binom{a}{a-j}\binom{a-j}{t}  \\
	&=\sum_{j=0}^{m}(-1)^j\binom{a}{t}\binom{a-t}{a-j-t} =\binom{a}{t}\sum_{j=0}^{m}(-1)^j \binom{a-t}{j}   \\
	&=(-1)^{m}\binom{a}{t}\binom{a-t-1}{m}.
\end{align*}

Thus (\ref{f9}), as a polynomial in $a$, holds for any $a\in \mathbb Z_{\ge m+t}$. So it holds for any $a\in \mathbb R$.
\ep

\ble\label{fl1}
For any $l\ge 0$ and  $p(z)\in K\left [ z \right ]$ with $deg \ p\le l$, one has
\begin{align}
	\sum_{t=0}^{l}(-1)^{l-t}\binom{z}{t}\binom{z-1-t}{l-t}p(t)=p(z). \label{f11}
\end{align}
\ele
\bp Fix some $z \in \left \{ 0,1,\cdots,l \right \}$.

As $z-1-t<l-t$, if $\binom{z-1-t}{l-t}\ne 0$ then $z-1-t<0$, i.e., $t>z-1$. If $\binom{z}{t}\ne0$ then $ t\le z$.

Then  $(-1)^{l-t}\binom{z}{t}\binom{z-1-t}{l-t}p(t)\ne 0$ only \  when \ $t=z,$ so the left hand side  of (\ref{f11}) equals $(-1)^{l-z}\binom{z}{z}\binom{z-1-z}{l-z}p(z)=p(z)$.

So (\ref{f11}) holds for $z=0,1,\cdots ,l$, then the lemma holds as $deg \ p\le l$.
\ep

Recall that
\be  \label{e60}
\binom{a}{n}\binom{n}{k}=\binom{a}{k}\binom{a-k}{n-k}, n,k\in\mathbb Z, n\ge k\ge0;a\in \mathbb R,
\ee and
\be  \label{e61}
\sum_{j=0}^{n}(-1)^j\binom{a}{j}=(-1)^n \binom{a-1}{n}, n\ge0, a\in \mathbb R.
\ee
For $(i,m)\in \mathbb Z^2$ with $m\ge i\ge1$,  let  
\[T(i,m)=\left \{ (j_1,\cdots, j_i)\in (\mathbb Z_{\ge1})^i |\sum _{t=1}^{i}j_t=m \right \}.\]
\ble\label{fl2}
For  $(i,l)\in \mathbb Z^2$ with $1 \le i \le l$,\quad let
$$\varphi (i,l)=\sum_{\left ( j_1,\cdots ,j_i \right )\in T(i,l) }\binom{\nu^{-1}}{j_1}\cdots \binom{\nu^{-1}}{j_i}.   $$
Then
\begin{align}
	\varphi (i,l)=\sum_{j=0}^{i-1}(-1)^j \binom{i}{j}\binom{(i-j)\nu^{-1}}{l}. \label{f5}
\end{align}
\ele
\bp   We prove it by induction on $i$.
%\hspace{5em} 

For $i=1,\quad \varphi(1,l)=\binom{\nu^{-1}}{l}$\quad for any $l \ge 1,$ then (\ref{f5}) holds.

For $m,i\ge1$, one has 
\begin{align}
	\sum _{t=0}^{m-1}(-1)^{t}\binom{i+1}{m-t}\binom{i+1-(m-t)}{t}  
	&=\sum_{t=0}^{m-1}(-1)^t \binom{i+1}{m}\binom{m}{m-t} \ \notag\ \  (by\ \ref{e60})	\\
	&=\binom{i+1}{m}\sum_{t=0}^{m-1}(-1)^t \binom{m}{t}\notag	\\
	&=(-1)^{m-1}\binom{i+1}{m}.  \ \ \ \ (by\  \ref{e61}) \label{e62}
\end{align}
Assume that $l>i\ge 1$,and (\ref{f5}) holds for all $\varphi(j,l)$ with $1\le j\le i.$ We will show that  $\varphi(i+1,l)$ satisfies (\ref{f5}).

Let $S=\{( j_1,\cdots ,j_{i+1})\in\mathbb Z_{\ge0}| j_1+\cdots +j_{i+1}=l \}.$ For $k=0,1,\cdots,i$, let
\[S_k=\{( j_1,\cdots ,j_{i+1})\in S| \text{exactly k elements in}\ j_1,\cdots ,j_{i+1}\ \text{are 0}\}.\]
One has
\begin{align*}
	&	\binom{(i+1)\nu^{-1}}{l}=\sum_{( j_1,\cdots ,j_{i+1})\in S  }\binom{\nu^{-1}}{j_1}\cdots \binom{\nu^{-1}}{j_{i+1}}\\
	&=(\sum_{( j_1,\cdots ,j_{i+1})\in S_0  }+\sum_{( j_1,\cdots ,j_{i+1})\in S_1  }+\cdots+\sum_{( j_1,\cdots ,j_{i+1})\in S_i })\binom{\nu^{-1}}{j_1}\cdots \binom{\nu^{-1}}{j_{i+1}}\\
	&=\binom{i+1}{0}\varphi(i+1,l)+\binom{i+1}{1}\varphi(i,l)+\cdots+\binom{i+1}{i}\varphi(1,l).
	\end{align*}
% \[\sum_{j=0}^{i}\binom{i+1}{j}\varphi (i+1-j,l)=\binom{(i+1)\nu ^{-1}}{l}. \]
Then
\begin{align}
	&\ \varphi (i+1,l) \notag=\binom{(i+1)\nu ^{-1}}{l}-\sum_{j=1}^{i}\binom{i+1}{j}\varphi(i+1-j,l) \notag\\
	&=\binom{(i+1)\nu^{-1}}{l}-\sum_{j=1}^{i}\binom{i+1}{j}\cdot \sum_{t=0}^{i-j}(-1)^t\binom{i+1-j}{t}\binom{(i+1-j-t)\nu^{-1}}{l} \notag  \\
	&=\binom{(i+1)\nu^{-1}}{l}-\sum_{j=1}^{i}\sum_{t=0}^{i-j}(-1)^t\binom{i+1}{j}\binom{i+1-j}{t}  \binom{(i+1-j-t)\nu^{-1}}{l} \notag \tag*{( let $ m=j+t$)}\\
	&=\binom{(i+1)\nu^{-1}}{l}-\sum_{m=1}^{i}\left [ \sum_{t=0}^{m-1}(-1)^t\binom{i+1}{m-t} \binom{i+1-(m-t)}{t}  \right ]\binom{(i+1-m)\nu^{-1}}{l} \notag\\  
	&=\binom{(i+1)\nu^{-1}}{l}-\sum_{m=1}^{i}(-1)^{m+1}\binom{i+1}{m}\cdot \binom{(i+1-m)\nu^{-1}}{l}  \ \ (by\  \ref{e62})\notag\\
	&=\sum_{m=0}^{i}(-1)^m\binom{i+1}{m}\binom{(i+1-m)\nu^{-1}}{l},  \notag 
\end{align}
so $\varphi(i+1,l)$ satisfies (\ref{f5}).
\ep

Now we prove Proposition \ref{fp1}.
\bp
By (\ref{e41}) (see also (\ref{e34})) we need to show that for $n\le 0, a,b\in L$,  \be\begin{aligned}\label{f13}  \delta_n(ab)=a\delta_n(b)+\sum_{n<l\le 0}\delta_l(a)\cdot\sum_{k=1}^{l-n}\binom{l}{k}[ \sum_{(j_1,\cdots,j_k)}^{j_1+\cdots+j_k=n+k-l}\delta_{j_1}\cdots \delta_{j_k}(b)]+\delta_n(a)b
\end{aligned}.\ee

First assume that $n\not\equiv 1\ (mod\ \nu)$. If some $\delta_l(a)\delta_{j_1}\cdots \delta_{j_k}(b)$ in right hand side  of  (\ref{f13}) with $l+j_1+\cdots+j_k=n+k$ is not identically 0, then $l+j_1+\cdots+j_k\equiv k+1\ (mod\ \nu)$, while $n+k\not\equiv k+1\ (mod\ \nu)$, which is impossible. So in the case   $n\not\equiv 1\ (mod\ \nu)$, $\delta_n=0$ and both sides of (\ref{f13}) are  identically 0, thus (\ref{f13})
is fulfilled.

From now on we assume that   $n=1-k\nu$ for some $k\ge 1$.

For $k\ge 1, a,b\in L$, one has 
\begin{align}
	D_k(ab) =\sum_{j=0}^{k}D_j(a) D_{k-j}(b), where\ D_j=\frac{1}{j!}(d/d\alpha)^j.\notag \end{align}
As $D_k={d_k}^{-1}\cdot \delta _{1-k\nu }$, one has \be	
\delta _{1-k\nu }(ab)=\sum_{j=0}^{k}\frac{d_k}{d_j\cdot d_{k-j}}\delta _{1-j\nu}(a) \delta _{1-(k-j)\nu}(b).  \label{f1}
\ee

Assume \quad  $\sum_{i=1}^{k}j_i=u$ with each $j_i\ge1$.
Then by (\ref{e64}),
\begin{align}
	\delta _{1-j_1\nu}\cdots \delta _{1-j_k\nu }=(\frac{\nu}{s} )^u\binom{\nu ^{-1}}{j_1} \cdots \binom{\nu ^{-1}}{j_k}(\frac{d}{d\alpha})^u. \label{f2}
\end{align}

We need to show that (\ref{f13}) holds for $n=1-k\nu$ for $k\ge 1$. As the case $k=1$ is clear, we need to prove the following identity for  $k\ge 2, a,b\in L$:

\[ \delta _{1-k\nu } (ab) = \left(a \delta _{1-k\nu }(b)+\delta _{1-k\nu }(a)b\right)+  \sum _{l=1}^{k-1}\delta _{1-l\nu }(a) \left [  \sum _{i=1}^{k-l}\binom{1-l\nu}{i}\sum _{(j_1,\cdots ,j_i)\in T(i,k-l)}\delta _{1-j_1\nu} \cdots \delta _{1-j_i\nu }(b)  \right ].  \]
{By (\ref{e64})} the right hand side  equals \[\left(a \delta _{1-k\nu }(b)+\delta _{1-k\nu }(a)b\right)+  \sum_{l=1}^{k-1}\delta _{1-l\nu }(a) \sum _{i=1}^{k-l}\binom{1-l\nu }{i} (\frac{\nu}{s} )^{k-l}\left (  \sum_{(j_1,\cdots ,j_i)\in T(i,k-l)}\binom{\nu^{-1} }{j_1}\cdots \binom{\nu ^{-1}}{j_i} \right )\left ( \frac{d}{d\alpha}  \right )^{k-l}(b). \]

Then by (\ref{f1}) one only need to show for $k\ge 2$ and $1 \le l \le k-1$,
\[	\sum_{i=1 }^{k-l}\binom{1-l\nu }{i}(\frac{\nu}{s} )^{k-l}\left ( \sum_{(j_1,\cdots ,j_i)\in T(i,k-l)}\binom{\nu^{-1}}{j_1}\cdots \binom{\nu^{-1}}{j_i}   \right )\left ( \frac{d}{d\alpha}  \right )^{k-l}(b)=\frac{d_k}{d_l\cdot d_{k-l}}\delta _{1-(k-l)\nu }(b),  \]
or,
\begin{align} \label{f3}
	&\sum_{i=1 }^{k-l}\binom{1-l\nu }{i}(\frac{\nu}{s} )^{k-l}\left ( \sum_{(j_1,\cdots ,j_i)\in T(i,k-l)}\binom{\nu^{-1}}{j_1}\cdots \binom{\nu^{-1}}{j_i}   \right )\left ( \frac{d}{d\alpha}  \right )^{k-l}=\frac{d_k}{d_l\cdot d_{k-l}}\delta _{1-(k-l)\nu }.   
\end{align}
As $ \delta _{1-(k-l)\nu }=d_{k-l}\cdot D_{k-l},$
right hand side  of (\ref{f3}) is \begin{align*}
	(\frac{\nu}{s} )^{k-l}\frac{\left [ \nu^{-1} \right ]_k }{\left [ \nu^{-1} \right ]_l }\cdot \frac{1}{(k-l)!}\left (  \frac{d}{d\alpha} \right ) ^{k-l}  
	=(\frac{\nu}{s} )^{k-l}\binom{\nu^{-1}-l}{k-l}\left ( \frac{d}{d\alpha}  \right ) ^{k-l},
\end{align*}
so (\ref{f3}) is equivalent to 
\begin{align}
	\sum_{i=1}^{k-l}\binom{1-l\nu}{i}\cdot \sum_{\left ( j_1,\cdots,j_i \right )\in T(i,k-l) } \binom{\nu ^{-1}}{j_1}\cdots \binom{\nu^{-1}}{j_i}=\binom{\nu^{-1}-l}{k-l}.  \label{f4}
\end{align}
By Lemma \ref{fl2}, the left hand side  of (\ref{f4}) is
\[\sum_{i = 1}^{k-l}\binom{1-l\nu}{i}\sum_{j=0}^{i-1}(-1)^j \binom{i}{j}\binom{(i-j)\nu^{-1}}{k-l}=\sum_{i=1}^{k-l}\sum_{j=0}^{i-1}(-1)^j\binom{1-l\nu}{i}\binom{i}{j}\binom{(i-j)\nu^{-1}}{k-l}.\]
So (\ref{f4}) is equivalent to 
\[ \sum_{i=1}^{k-l}\sum_{j=0}^{i-1}(-1)^j\binom{1-l\nu}{i}\binom{i}{j}\binom{(i-j)\nu^{-1}}{k-l}=\binom{\nu^{-1}-l}{k-l}.\]

Replacing $k-l$ by $u$, it is 
\begin{align}
	\sum_{i=1}^{u}\sum_{j=0}^{i-1}(-1)^j\binom{1-l\nu}{i}\binom{i}{j}\binom{(i-j)\nu^{-1}}{u}=\binom{\nu^{-1}-l}{u}, \ \ u\ge 1.  \label{f8}
\end{align}

The left hand side  of (\ref{f8}) equals 
\begin{align*}
	& \sum_{i=1}^{u}\sum_{j=0}^{i-1}(-1)^j \binom{1-l\nu }{j}\binom{1-l\nu -j}{i-j}\binom{(i-j)\nu^{-1}}{u}\ \ \ \ \text{(Let $t=i-j$)}\\
	&=\sum_{t=1}^{u}\left ( \sum_{j=0}^{u-t} (-1)^j\binom{1-l\nu}{j}\binom{1-l\nu-j}{t}  \right )\binom{t\nu^{-1}}{u} \\
	&=\sum_{t=1}^{u}(-1)^{u-t}\binom{1-l\nu }{t}\binom{-l\nu -t}{u-t}\binom{t\nu^{-1}}{u}   \ \ (\text{by Lemma } \ref{fl3} \text{ with $m=u-t$}) \\
	&=\sum_{t=0}^{u}(-1)^{u-t}\binom{1-l\nu }{t}\binom{-l\nu-t}{u-t}\binom{t\nu^{-1}}{u}.       
\end{align*}

So (\ref{f8}) is equivalent to 
\begin{align}
	\sum_{t=0}^{u}(-1)^{u-t}\binom{1-l\nu }{t}\binom{-l\nu-t}{u-t}\binom{t\nu^{-1}}{u} 
	=\binom{\nu^{-1}-l}{u}, \ \ u\ge 1.  \label{f12}
\end{align}
Set $z=1-l\nu,\ p(z)=\binom{\nu^{-1}z}{u}$ in (\ref{f11}),  we get (\ref{f12}).% The proof of Proposition \ref{fp1} is completed.

Now we prove the formula (\ref{f15}) in Proposition \ref{fp1}.

By (\ref{e21}), for $b\in K(\alpha), n\in \mathbb Z$ one has
\begin{align*}
T^nb&=bT^n+\binom{n}{1}\sum_{k\geq1}\frac{d_k}{k!}b^{(k)}T^{n-k\nu}+\binom{n}{2}\sum_{k_1,k_2\geq1}\frac{d_{k_1}}{k_1!}\frac{d_{k_2}}{k_2!}b^{(k_1+k_2)}T^{n-(k_1+k_2)\nu}+\cdots\\
&=\sum_{k\geq0}\lambda_n^kb^{(k)}T^{n-k\nu}, \ \ \lambda_n^0=1.
\end{align*}
%.So we can assume that:
%\[T^nb=\sum_{k\geq0}\lambda_n^kb^{(k)}T^{n-k\nu}\]
%at the same time, where
(Note that $\lambda_{n}^k$ is not the $k$-th power of $\lambda_{n}$.) Comparing the coefficients of $T^{n-k\nu}$, one finds that if $k>0$ then
\begin{align}
	\lambda_n^k&=\sum_{r=1}^{k}\left[\binom{n}{r}\cdot\sum_{\substack{s_1+s_2+...+s_r=k\\s_1,s_2,...,s_r\geq 1}}\frac{d_{s_1}}{s_1!}\frac{d_{s_2}}{s_2!}\cdots\frac{d_{s_r}}{s_r!} \right]=\sum_{r=1}^{k}\binom{n}{r}\cdot \left(\dfrac{\nu}{s}\right)^k\cdot \sum_{\substack{s_1+...+s_r=k\\s_1,...,s_r\geq 1}}\binom{\nu^{-1}}{s_1}\cdots \binom{\nu^{-1}}{s_r}	
	 \notag\\
	&=\left(\dfrac{\nu}{s}\right)^k\sum_{r=1}^k\sum_{j=0}^{r-1}(-1)^j\binom{n}{r}\binom{r}{j}\binom{(r-j)\nu^{-1}}{k}\ \ \ \ (by\ lemma\ 4.4)\notag\\
	&=\left(\dfrac{\nu}{s}\right)^k\sum_{i=1}^k\sum_{j=0}^{k-i}(-1)^j\binom{n}{i+j}\binom{i+j}{j}\binom{i\nu^{-1}}{k}\ \ \ \ \ \ ( i=r-j)\notag\\
	&=\left(\dfrac{\nu}{s}\right)^k\sum_{i=1}^k\sum_{j=0}^{k-i}(-1)^j\binom{n}{i}\binom{n-i}{j}\binom{i\nu^{-1}}{k}\notag\\
	&=\left(\dfrac{\nu}{s}\right)^k\sum_{i=1}^k\left[\binom{n}{i}\binom{i\nu^{-1}}{k}\sum_{j=0}^{k-i}(-1)^j\binom{n-i}{j}\right]\notag\\
	&=\left(\dfrac{\nu}{s}\right)^k\sum_{i=1}^k(-1)^{k-i}\binom{n}{i}\binom{n-i-1}{k-i}\binom{i\nu^{-1}}{k}.\ \ \ \ (by\ (\ref{e61})).\notag
\end{align}

\ep

Note that this proposition also holds for $L=K((\alpha^{-1}))$.

Assume  $ \left ( r , s  \right ) \in \mathbb Z^2, s \ne0, \nu =r+s, \nu>0. $
Let $\alpha=\alpha_{ r , s}, T=T_{r , s}$ and  $ \widetilde{D}_{r, s}=L((T^{-1}; \widetilde{\delta }_{r, s })) $ with $ L=K\left ( \alpha  \right ). $  We also write $ \widetilde{D}_{r, s } =K(\alpha )\left ( \left (T^{-1}; \widetilde{\delta }_{r, s }  \right )  \right ) $.
Next we extend the maps $\mathsf{v}_{r, s }$ and $\mathsf{f}_{r, s }$ to $ \widetilde{D}_{r, s }$.

Let \[ \mathsf{v}_{r, s }:\  \widetilde{D}_{r, s } \to \mathbb Z_{-\infty }, \ 0\mapsto -\infty; T\mapsto 1; z=\sum_{i\le n}a_i\left ( \alpha  \right ) T^i\mapsto n.  \]
Note that when we write some element in $\widetilde{D}_{r, s }$ as $z=\sum_{i \le n}a_i(\alpha )T^i$, we will always assume that $a_n\ne 0$. Similar conventions will be used henceforth. It is clear that $ \mathsf{v} _{r, s }$ is a  degree map on $ \widetilde{D}_{r, s } $.

One has $ \mathsf{v}_{r, s }\left ( T \right )=1$ and $\mathsf{v}_{r, s }\left ( a(\alpha ) \right )=0$ for $ 0\ne a(\alpha )\in K(\alpha )$.

Let $\widetilde{D}_{r, s }^{\le i}= \mathsf{v}_{r, s }^{-1}(\mathbb Z_{\le i}\cup \{-\infty\} )$ for $i\in \mathbb Z$. There is a natural filtration on $ \widetilde{D}_{r, s }$
\[\cdots\subset \widetilde{D}_{r, s }^{\le i-1}\subset \widetilde{D}_{r, s }^{\le i}\subset \widetilde{D}_{r, s }^{\le i+1}\subset\cdots\]
By (\ref{g2}) one has for $i,j\in\mathbb Z$, \[[f(\alpha)T^i,g(\alpha)T^j]=f(\alpha)[T^i,g(\alpha)]T^j+g(\alpha)[f(\alpha),T^j]T^i\in K(\alpha)\cdot T^{i+j-\nu},\] thus \[[\widetilde{D}_{r, s }^{\le i},\widetilde{D}_{r, s }^{\le j}]\subseteq \widetilde{D}_{r, s }^{\le i+j-\nu}.\]
As $\nu\ge1$, the associated graded algebra $gr(\widetilde{D}_{r, s })=\oplus_{i\in \mathbb Z}\ gr_i(\widetilde{D}_{r, s }), gr_i(\widetilde{D}_{r, s })=\widetilde{D}_{r, s }^{\le i}/\widetilde{D}_{r, s }^{\le i-1} $, is commutative.

% Let $ \widetilde{V}_{r, s } :={v_{r, s }}^{-1}(\mathbb Z_{\le0}\cup \{-\infty\} )$ be the valuation  ring of $ \widetilde{D}_{r, s }$ with respect to $ v_{r, s }$. $ \widetilde{V}_{r, s } $ is a non-commutative PID and $ I={v_{r, s }}^{-1}(\mathbb Z_{<0}\cup\{ -\infty \})$ is the unique maximal ideal of $ \widetilde{V}_{r, s } $. The residue field is $ \widetilde{V}_{r, s } /I\simeq K(\alpha )$. It is clear that the associated graded algebra $gr( \widetilde{V}_{r, s })=\oplus_{i\in \mathbb Z_{\le0}}(\widetilde{D}_{r, s }^{\le i}/\widetilde{D}_{r, s }^{\le i-1} )$ is also commutative. So 
% $ \widetilde{V}_{r, s }$ satisfies all the conditions of  a Henselian ring (see  \cite{a} for the definition of non-commutative  Henselian rings), thus one has

Now we show that (See  \cite{a} for the definition of non-commutative  Henselian rings.)
  \begin{prop}\label{f35}
 The ring $ \widetilde{V}_{r, s }$ is a Henselian ring.
 \end{prop}

 This result is a consequence of  the following lemma.
 
 	Assume $D$ is a division ring, $u:D\to \mathbb{Z}_{-\infty}$ is a degree map. Let $V=D^{\leq 0}$ be its valuation ring. Then $I = D^{\leq -1}$ is the unique maximal ideal of $V$, thus $V$ is a local ring. The degree map $u$ induces a filtration on $D$:
 \[\cdots \subset D^{\leq i-1}\subset D^{\leq i}\subset D^{\leq i+1}\subset \cdots ,  i\in \mathbb{Z},\] with the associated graded algebra $gr(D) =\bigoplus_{i\in \mathbb Z} gr_{i}(D), gr_{i}(D) =  D^{\leq i}/ D^{\leq i-1}$.
 \begin{lemma}\label{e69}
 	
 	Assume 
 	that 
 	
 	(1) $D$ is a division ring complete with respect to  $u$;
 	
 	(2) $[D^{\leq i},D^{\leq j}] \subseteq D^{\leq i+j-1}$, $\forall i,j\in \mathbb{Z}$. 
 	
 	Then the valuation ring $V$ is a Henselian ring.
 \end{lemma}
 \bp
 
 One has that  $V$ is a local ring with the unique maximal ideal of $I$,  $\cap_{n\in \mathbb Z} I^n = \{0\}$. 	The degree map $u$ restricts to the degree map $u|_V$ on $V$, which induces the corresponding filtration on $V$. 
 By (2), $gr(D)$ is commutative, thus 	$gr(V)$ is also commutative. It follows that the residue field $V/I$ is also commutative.

 As  $D$ is complete with respect to  $u$, 
 $V$ is also complete with respect to   $u|_V$.
 
 By Theorem 2.1 of \cite{a}, $V$ is Henselian.\ep

Let \[\mathsf{f}_{r, s }:\widetilde{D}_{r, s } \to K(X, Y^{1/s }), z=\sum_{i \le n}a_i(\alpha )T^i\mapsto a_n(XY^{-r /s })Y^{n/s }, 0\mapsto 0. \]

Its image is $\mathsf{f}_{r, s } (\widetilde{D}_{r, s })=\bigcup_{n\in \mathbb Z}K(XY^{-r /s })Y^{n/s }$, and
$ \mathsf{f}_{r, s } (\widetilde{D}_{r, s })\setminus\{0\}$ is a multiplicative group.
\begin{prop}
For	$ z_1, z_2 \in \widetilde{D}_{r, s }$, one has
	
	(1)\ $\mathsf{f}_{r, s }(z_1z_2)=\mathsf{f}_{r, s }(z_1)\mathsf{f}_{r, s }(z_2). $
	
	(2)\ For $ \lambda  \in K,  \mathsf{f}_{r , s} (\lambda )=\lambda;  \mathsf{f}_{r, s }(\lambda z)=\lambda \mathsf{f}_{r , s}(z). $
	
	(3)$ \mathsf{f}_{r, s }:\widetilde{D}_{r, s } \to K(X, Y^{1/s })   $
	is a multiplicative homomorphism.
	In particular,
	\[ \mathsf{f}_{r, s }(z^{-1})={\mathsf{f}_{r, s }(z)}^{-1}, z\ne 0. \]
	
	(4) If $ \mathsf{v}_{r, s }( z_1 )=\mathsf{v}_{r, s }( z_2)=\mathsf{v}_{r, s }( z_1+z_2 )$, then $\mathsf{f}_{r, s }( z_1+z_2)=\mathsf{f}_{r, s }( z_1 )+\mathsf{f}_{r, s }( z_2 ).  $ \\
\end{prop}
\begin{proof}
	(1) Assume $z_1=\sum_{i\le n}a_i ( \alpha ) T^i, z_2=\sum_{j\le m}b_j ( \alpha ) T^j$. Then \[z_1z_2=a_n(\alpha)b_m ( \alpha ) T^{n+m}+\sum_{i< n+m}c_i(\alpha)T^i,\] thus $\mathsf{f}_{r, s }(z_1z_2)=\mathsf{f}_{r, s }(z_1)\mathsf{f}_{r, s }(z_2)$.
	
	(2) is clear.
	
	(3) By (2) $\mathsf{f}_{r, s }(1)=1$. Then it follows from  (1) that $ \mathsf{f}_{r, s }$
	is a multiplicative homomorphism.
	
	(4) is clear.
	
\end{proof}

Now we embed $D_1$ into $\widetilde{D}_{r, s } $. By (\ref{g1}) one has 
\[[T,\alpha]=s^{-1}T^{1-\nu},\ so\]  \[ \left [ T^s,\alpha T^r\right ]=[ T^s,\alpha ]T^r=s T^{s-1}[T,\alpha]T^r=s T^{s-1}(s^{-1}T^{1-\nu})T^r=T^{s-\nu+r}=1.\] 
So there exists a unique $K$-algebra homomorphism
\[ \eta :A_1 \to \widetilde{D}_{r, s }, \textbf{q} \mapsto  T^s, \textbf{p}\mapsto\alpha T^r, \]
which is injective as $ A_1$ is simple . Since $ \widetilde{D}_{r, s }$is a division ring , 
$ \eta $ can be extented to an injective $K$-homomorphism of division algebras
\[ \eta :D_1 \to \widetilde{D}_{r, s }, \ zw^{-1} \mapsto \eta(z)\eta(w )^{-1}. \]
Henceforth we will identify $ D_1$ as a subalgebra of $ \widetilde{D}_{r, s }$
by $\eta $, and we denote \[  T=\textbf{q}^{1/s }, \alpha=\textbf{p}T^{-r}=\textbf{pq}^{-r /s }.\]

Let  $ \widehat {D}_{r, s }$ be the completion of the $K$-division algebra $D_1$ 
in  $\widetilde{D}_{r, s }$.  As $ \widetilde{D}_{r, s }$ is complete,  
$ \widehat {D}_{r, s }$ is a complete $K$-division subalgebra of $ \widetilde{D}_{r, s }$. Thus one has \[A_1\subset D_1\subset \widehat {D}_{r, s }\subseteq \widetilde{D}_{r, s }.\]

The map $\mathsf{f}_{r, s } $ restricts to a multiplicative homomorphism
$ \mathsf{f}_{r, s } : D_1 \to K(X, Y^{1/s }). $ One has
\[\mathsf{f}_{r, s }(q)=\mathsf{f}_{r, s }(T^s)=Y,  \mathsf{f}_{r, s }(p)=\mathsf{f}_{r, s }(\alpha T^r)=XY^{-r/s}Y^{r/s}=X. \]
If $ z\in A_1, \mathsf{f}_{r, s } (z) \in K[X, Y]$. \ Thus if $ z, w \in A_1$\ with $ w \ne 0$, then $\mathsf{f}_{r, s }(zw^{-1})\in K(X, Y).  $So we have the following commutative diagram:
\begin{center}
	\begin{tikzcd}
		& D_1  \arrow[r,   "\mathsf{f}_{r, s} "]  \arrow[d]  &K(X, Y)\arrow[d] \\
		& \widetilde{D}_{r, s } \arrow[r,  "\mathsf{f}_{r, s} "] &K(X, Y^{1/s})
	\end{tikzcd}
\end{center}
Note that if $(r,s)=(0,1)$, then $ \widetilde{D}_{0, 1}=\widehat{D}_{0, 1}=K(\textbf{p})((\textbf{q}^{-1}; \delta ))  $.

	It is clear that $\mathsf{v}_{r, s}$ restrict to a degree map  on 
	$\widehat{D}_{r, s }$, 
	the associated graded algebra $gr( \widehat{D}_{r, s })$ of $ \widehat{D}_{r, s }$ with respect to this degree map is commutative, and the valuation  ring
	$ \widehat{V}_{r, s } =\widehat{D}_{r, s}\cap \widetilde{V}_{r, s}.$ 
	By Lemma \ref{e69}, one has 
	\bco
The valuation  ring $\widehat{V}_{r, s }$ of $\widehat{D}_{r, s }$	is a Henselian ring.	\\[1mm]
\eco

Next we give the result parallel to Proposition \ref{fp1}.

Assume that $(r,s )\in \mathbb Z^2$ with $r \ne 0$, where $r,s$ are not required to be coprime. Let $\nu =r +s$. Assume $\nu>0$.

Let $d'_k=(\nu/r)^k[\nu^{-1}]_k$ for $k>0$, and $d'_0=1$. Let $L=K(\beta)$ with $\beta$ being an indeterminant, $\eta_i\in End_{K-lin}(L)$ be defined as follows:
$$\eta _i=
\left\{
\begin{array}{ll}
	\frac{d'_k}{k!} (\frac{d}{d\beta})^k, \ \ & i=1-k\nu \text{ with } k\ge 0;\\
	0, & i\le 1 \text{ and } i\not\equiv 1\ (mod\ \nu).
\end{array}
\right.$$

\begin{coro}   \label{fp2}  
		 The set of linear operators $\eta_i,i\in \mathbb Z_{\le0}$, satisfy (\ref{e41}), so there exists 
	a multiplication  $\Tilde{\eta} $ on  $L_\textbf{r}((T^{-1}))$ such that
	 
	\[f\cdot T=T\cdot f+\sum_{k=1}^{\infty } T^{1-k\nu }\cdot \eta _{1-k\nu }(f)
	=T\cdot f+r^{-1}T^{1-\nu}f'+\frac{1}{2!}\frac{1-\nu}{r^2}T^{1-2\nu}f^{(2)}+\cdots,f\in L,\]   and $L_\textbf{r}((T^{-1};\Tilde{\eta}))$ is a deformed Laurent series ring, with the multiplication \[\sum_{i\le m} T^ia_i \cdot \sum_{j\le n} T^jb_j= \sum_{i\le m,j\le n}  T^ia_i\cdot T^j b_j,\ where\]  \begin{align}\label{} & T^ia_i\cdot T^jb_j=\sum_{k\geq 0}\lambda_j^kT^{j-k\nu+i}a_i^{(k)}b_j,
	\end{align}and $\lambda_j^k$ is defined as in (\ref{f15}).
	
\end{coro} 
Denote this ring by $\check{D}_{r,s}$, and also by  $K(\beta)_\textbf{r}\left ( \left (T^{-1};\Tilde{\eta}_{r,s}\right )  \right )$. 

It satisfies that for $f\in L, n\in \mathbb Z$, 
$$[f,T^n]=\binom{n}{1}\sum_{k>0}\frac{d'_k}{k!} T^{n-k\nu}\cdot (\frac{d}{d\beta})^kf + \binom{n}{2}\sum_{k_1,k_2>0}\frac{d'_{k_1}d'_{k_2}}{k_1!k_2!}T^{n-(k_1+k_2)\nu}\cdot(\frac{d}{d\beta})^{k_1+k_2} f + \cdots$$
In particular,
$$[\beta,T^n]=nr^{-1}T^{n-\nu}.$$
Then 
$$[T^s\beta,T^r]=T^s[\beta,T^r]=T^srr^{-1}T^{r-\nu}=1.$$
So there exists a unique $K-$algebra homomorphism
$$\xi:A_1\rt K(\beta)_\textbf{r}((T^{-1};\Tilde{\eta}_{r,s})),\ \ \textbf{p}\mapsto T^r, \textbf{q}\mapsto T^s \beta,$$
which can be extended to an injective $K$-homomorphism from $D_1$ to $\check{D}_{r,s}$. We identify $D_1$ as a subalgebra of $K(\beta)_\textbf{r}((T^{-1};\Tilde{\eta}_{r,s}))$ by $\xi$,   and we denote \[  T=\textbf{p}^{1/r }, \beta=\textbf{p}^{-s/r}\textbf{q}.\]

Analogously, one defines \[ u_{r, s }:\  \check{D}_{r, s } \to \mathbb Z_{-\infty }, \ 0\mapsto -\infty; T\mapsto 1; z=\sum_{i\le n}T^ib_i\left ( \beta  \right ) \mapsto n,  \] which is a degree map on $ \check{D}_{r, s }$. The associated graded algebra $gr( \check{D}_{r, s })$ of $ \check{D}_{r, s }$ with respect to $u_{r, s }$ is commutative, and the valuation  ring
 $ \check{V}_{r, s } :={u_{r, s }}^{-1}(\mathbb Z_{\le0}\cup \{-\infty\} )$ of $ \check{D}_{r, s }$ is a Henselian ring.
 
 It is clear that the completion of $D_1$ in  $ \check{D}_{r, s }$ with respect to $u_{r,s}$ is isomorphic to $\widehat{D}_{r,s}$.

 \section{Topological generators of $\widehat{D}_{r,s}$}
 \setcounter{equation}{0}\setcounter{theorem}{0}
 
 Let $K$ be a field of characteristic 0. Let $A_1=A_1(K), D_1=D_1(K)$.

Assume that $(r, s)\in \mathbb{Z}^2$, $s \ne0$, and $r $ 
, $s $ are coprime. Let $\nu = r + s > 0$. Choose $i , j\in \mathbb{Z}$ such that $r i + s j =1$. 
Let $$T_0= \textbf{p}^i\textbf{q}^j,\ \ \alpha_0 = \textbf{p}^s \textbf{q}^{-r}.$$ One has
\be
\binom{s\ \ \ i}{-r\ \ j} \binom{j\ \ -i}{r\ \ s}=	\binom{1\ \ 0}{0\ \ 1},
\ee
which can be applied to simplify some of the calculations later.

One has  $\mathsf{v}_{r, s}(T_0)= 1,\mathsf{v}_{r, s}(\alpha_0)= 0 .$
Let $$R = \left\{\ \sum_{i\leq n}a_i(\alpha_0)T_0^i\ |\ a_i\in K(X), n\in \mathbb Z.\right\}\cup \{0\}.$$
One has $R\subseteq \widetilde{D}_{r, s}$, and it is clear that $R$ is complete. We will show that 	$R = \widehat{D}_{r, s}$, thus $\alpha_0,T_0$ can be viewed as a pair of topological generators of $\widehat{D}_{r, s}$.

\begin{lemma}
	Assume $z_i=\textbf{p}^{l_i}\textbf{q}^{k_i}\in A_1$, $i=1,\cdots,t$ and $t\ge 2$. Let $l=\sum_{i=1}^t l_i, k=\sum_{i=1}^t k_i$. Then \[\Pi_{i=1}^t z_i=\textbf{p}^l\textbf{q}^k+\sum_{m\ge1} \mu_m \textbf{p}^{l-m}\textbf{q}^{k-m}, \mu_m\in \mathbb Q.\]
\end{lemma}
This follows from (\ref{f32}).

\begin{lemma}\label{h6}
$R$ is a complete sub-division-ring of $\widetilde{D}_{r,s}$.
\end{lemma}
\begin{proof}
We only need to prove that $R$ is closed under multiplication and inversion. First we show that $R$ is closed under multiplication.

One has
\begin{align}
	\textbf{p}^i\textbf{q}^j\cdot \textbf{p}^u\textbf{q}^t=	\textbf{p}^{i+u}\textbf{q}^{j+t}+\sum_{k\ge1}\frac{[j]_k[u]_k}{k!}\textbf{p}^{i+u-k}\textbf{q}^{j+t-k}, i,j,u,t\in \mathbb Z.
\end{align}

Then for $k,l\in\mathbb{Z}$,
$$\alpha_0^kT_0^l= \textbf{p}^{ks+li} \textbf{q}^{-kr+lj}+\sum_{m\geq 1}\lambda_{m}\textbf{p}^{ks+li-m}\textbf{q}^{-kr+lj-m}, $$
$$T_0^l\alpha_0^k= \textbf{p}^{ks+li} \textbf{q}^{-kr+lj}+\sum_{m\geq 1}\mu_{m}\textbf{p}^{ks+li-m}\textbf{q}^{-kr+lj-m} , \lambda_{m},\mu_{m}\in \mathbb{Q}.$$
Then
\begin{equation}\label{h1}
	T_0^l\alpha_0^k= \alpha_0^kT_0^l +\sum_{m\geq 1}\xi_{m}\textbf{p}^{ks+li-m}\textbf{q}^{-kr+lj-m}, \xi_{m}=\mu_m-\lambda_{m}.
\end{equation}

As $$\alpha_0^{-(j-i)}T_0^{-\nu}=\textbf{p}^{-1}\textbf{q}^{-1}+\sum_{m\ge1}\mu_m \textbf{p}^{-1-m}\textbf{q}^{-1-m},$$ one has
\[\alpha_0^{k-m(j-i)}T_0^{l-m\nu}=\textbf{p}^{ks+li-m}\textbf{q}^{-kr+lj-m}+\sum_{s\ge1} \nu_s \textbf{p}^{ks+li-m-s}\textbf{q}^{-kr+lj-m-s},\]
and \begin{equation}\label{h2}	\textbf{p}^{ks+li-m}\textbf{q}^{-kr+lj-m}=\alpha_0^{k-m(j-i)}T_0^{l-m\nu}-\sum_{s\ge1} \nu_s \textbf{p}^{ks+li-m-s}\textbf{q}^{-kr+lj-m-s}, \nu_s\in \mathbb{Q}. \end{equation}
Substitute (\ref{h2}) into (\ref{h1}) for $m=1,2,3,\cdots,$ successively, one has

$$T_0^l\alpha_0^k= \alpha_0^kT_0^l +\sum_{m\geq 1}\epsilon_{m}\alpha_0^{k-m(j-i)}T_0^{l-m\nu} , \epsilon_{m}\in\mathbb{Q}.$$
Thus for any $f\in K[X], n\in \mathbb Z$, there exist $f_{n,m}\in K(X)$ such that 
\be T_0^nf(\alpha_0)=f(\alpha_0)T_0^n+\sum_{m\geq 1}f_{n,m}(\alpha_0)T_0^{n-m\nu}. \label{heq1}\ee
Now assume $f\ne 0$. Then multiply $ f(\alpha_0)^{-1}$ on both sides of (\ref{heq1}) and reorder the terms, one has
\be\label{h4} T_0^nf(\alpha_0)^{-1}=f(\alpha_0)^{-1}T_0^n+\sum_{m\geq 1}g_{n,m}(\alpha_0)T_0^{n-m\nu}f(\alpha_0)^{-1},\ee  where $g_{n,m}\in K(X),g_{n,m}(\alpha_0)=-f(\alpha_0)^{-1}f_{n,m}(\alpha_0)$. Then one has
\be\label{h5}T_0^{n-m\nu}f(\alpha_0)^{-1}=f(\alpha_0)^{-1}T_0^{n-m\nu}+\sum_{k\geq 1}g_{n-m\nu,k}(\alpha_0)T_0^{n-(m+k)\nu}f(\alpha_0)^{-1}, \ee substitute  $T_0^{n-m\nu}f(\alpha_0)^{-1}$ in (\ref{h4}) by the right hand side  of (\ref{h5}), one has
\begin{align*}
	T_0^nf(\alpha_0)^{-1}&=f(\alpha_0)^{-1}T_0^n+\sum_{m\geq 1}g_{n,m}(\alpha_0)\left(f(\alpha_0)^{-1}T_0^{n-m\nu}+\sum_{k\geq 1}g_{n-m\nu,k}(\alpha_0)T_0^{n-(m+k)\nu}f(\alpha_0)^{-1} \right) \\
	&=f(\alpha_0)^{-1}T_0^n+\sum_{m\geq 1}g_{n,m}(\alpha_0)f(\alpha_0)^{-1}T_0^{n-m\nu}+ \sum_{m,k\geq 1}g_{n,m}(\alpha_0)g_{n-m\nu,k}(\alpha_0)T_0^{n-(m+k)\nu}f(\alpha_0)^{-1}.
\end{align*}
Then substitute  $T_0^{n-(m+k)\nu}f(\alpha_0)^{-1}$ in the right hand side  of above equation by the right hand side  of (\ref{h5}) (with $m$ replaced by $m+k$), and		
continue this process. As $\mathsf{v}_{r, s}(T_0)=1$ and $\nu>0$, eventually one gets

\be T_0^nf(\alpha_0)^{-1} =f(\alpha_0)^{-1}T_0^n+\sum_{m\geq 1}a_{n,m}(\alpha_0)T_0^{n-m\nu}, a_{n,m}\in K(X). \label{heq2}\ee
By (\ref{heq1}) and (\ref{heq2}) ,we know that for any $h\in K(X)$ and $n\in \mathbb Z$, there exist $h_{n,m}\in K(X)$ such that
$$T_0^nh(\alpha_0) = h(\alpha_0)T_0^n +\sum_{m\geq 1}h_{n,m}(\alpha_0)T_0^{n-m\nu} \in R .$$
Then for any $f,g \in K(X), i,j\in\mathbb Z,$
$$f(\alpha_0)T_0^i \cdot g(\alpha_0)T_0^j = f(\alpha_0)g(\alpha_0)T_0^{i+j}+\sum_{m\geq 1}f(\alpha_0)g_{i,m}(\alpha_0)T_0^{i+j-m\nu} \in R.$$
For any $z,w\in R^\times$, $z=\sum_{i\leq n}a_i(\alpha_0)T_0^i, w=\sum_{j\leq m}b_j(\alpha_0)T_0^j,$ one has
$$zw = \sum_{k\le m+n}\sum_{i+j=k}a_i(\alpha_0)T_0^i\cdot b_j(\alpha_0)T_0^j= \sum_{k\le m+n}\sum_{i+j=k}\left( a_i(\alpha_0)b_j(\alpha_0)T_0^k+\sum_{t\ge1} a_i(\alpha_0)b_{j,i,t}(\alpha_0)T_0^{k-t\nu}\right)  \in R .$$
So $R$ is closed under multiplication.

Assume $z\in R^\times$. Then
$$z= \sum_{i\leq n}a_i(\alpha_0)T_0^i = a_n(\alpha_0)T_0^n(1+w), w = \sum_{i\leq n-1}T_0^{-n}a_n(\alpha_0)^{-1}a_i(\alpha_0)T_0^i\in R. $$
Then $\mathsf{v}_{r, s}(w)\le -1$. Let
$$u = (1+ \sum_{i\geq 1}(-1)^iw^i )T_0^{-n}a_n(\alpha_0)^{-1}\in R.$$
Then
$u=z^{-1}$.
Thus $R$ is closed under inversion.
\end{proof}
\begin{prop}
$R = \widehat{D}_{r, s}$.
\end{prop}
\begin{proof}
As $\alpha_0,T_0 \in D_1 , a_i(\alpha_0)T_0^i \in D_1$. As $\widehat{D}_{r,s}$ is a ring complete with respect to $\mathsf{v}_{r,s}$, for any $z = \sum_{i\leq n}a_i(\alpha_0)T_0^i \in R,$ one has $z\in \widehat{D}_{r,s}.$ Thus $R\subseteq \widehat{D}_{r, s}.$

On the other hand, 
$$\alpha_0^j T_0^r =\textbf{p}+ \sum_{m\geq 1}\lambda_{m}\textbf{p}^{1-m}\textbf{q}^{-m}, \alpha_0^{-i} T_0^s =\textbf{q}+ \sum_{m\geq 1}\mu_{m}\textbf{p}^{1-m}\textbf{q}^{-m}, \lambda_m,\mu_m\in \mathbb{Q}.$$
As in the proof of Lemma \ref{h6},one has 
$$\textbf{p}= \alpha_0^jT_0^r +\sum_{m\geq 1}\omega_m\alpha_0^{j-m(j-i)}T_0^{r-m\nu}, \textbf{q}= \alpha_0^{-i}T_0^s +\sum_{m\geq 1}\xi_m\alpha_0^{-i-m(j-i)}T_0^{s-m\nu} , \omega_m,\xi_m\in \mathbb{Q}.$$
So $\textbf{p},\textbf{q}\in R$. As $R$ is a skew field complete with respect to $\mathsf{v}_{r,s}$, one has $\widehat{D}_{r, s}\subseteq R$.
\end{proof}

\section{$\widehat{D}_{r, s}$ satisfies  the commutative centralizer	condition}
\setcounter{equation}{0}\setcounter{theorem}{0}
%label D

Recall that $D = L((T^{-1};\widetilde{\delta}))$ is a deformed Laurent series ring with the standard degree map $deg$ as defined in (\ref{e39}). Assume $char\ L=0$. Let $D = L((T^{-1};\widetilde{\delta}))$. Then
 $K= \cap_{i\leq 0}\text{Ker}(\delta_{i})$ is a field of characteristic 0, and is the center of $D$. Assume that there exists $i\in \mathbb Z_{\leq 0}$ with $\delta_i \neq 0$. Let $i_0 = \max\{i\in \mathbb Z_{\leq 0}| \delta_i \neq 0\}$. Then $\delta_{i_0}\in \text{Der}_K(L)$. For $z\in D$, let $C(z)$ be the centralizer of $z$ in $D$.
 
Recall that  a noncommutative algebra $A$ is said to satisfy the commutative centralizer
 condition (the ccc, for short) if the centralizer of each element not in the center of $A$ is a commutative algebra.

\begin{prop}\label{Dp1}
Assume $K=\overline{K}$ (i.e. $K$ is algebraically closed) and $\text{Ker}(\delta_{i_0}) = K$.  Assume $D = L((T^{-1};\widetilde{\delta}))$ and $z\in D\backslash K$. Then 

1. If $z\in L\backslash K$, then $C(z) = L$.

2. If $z\in D$ and $deg(z)\neq 0$, then $C(z) = K((v^{-1}))$, where $v$ is an element in $C(z)$ with the least positive degree. One has
$[C(z):K((z^{-1}))] = |\frac{deg(z)}{deg(v)}|$ if $deg(z) > 0$, and $[C(z):K((z))] =  |\frac{deg(z)}{deg(v)}|$ if $deg(z) < 0$.

3. If $z \in D$ and $deg(z) =0$, then write $z = a_0 + a_{-1}T^{-1} + \cdots$.

3.1. If $a_0\in K$, then $C(z) = C(z-a_0)$, $z-a_0$ is in Case 2.

3.2. If $a_0\in L\backslash K$, then $C(z)\rightarrow L, w = b_0 + b_{-1}T^{-1}+\cdots\mapsto b_0$, is a ring isomorphism.

So $D$ satisfies the ccc. For all	$z\in D\backslash K$, $C(z)$ is a complete subfield of D.
\end{prop}

\bp
1. This follows from Lemma \ref{e70}. 

2. Assume  $z = \sum_{i\leq n}a_iT^i$, and $w = \sum_{j\leq m}b_jT^j\in C(z)$.
One has
\begin{align}
a_iT^i \cdot b_jT^j = a_ib_jT^{i+j} + a_i \binom{i}{1}\delta_{i_0}(b_j)T^{i+j-1+i_0} + lower\ degree\ terms ,\ \ and \label{De1}\\		
b_jT^j \cdot a_iT^i = a_ib_jT^{i+j} + b_j\binom{j}{1}\delta_{i_0}(a_i)T^{i+j-1+i_0}+lower\ degree\ terms.\label{De2}
\end{align}
Then 
\be \label{D2}	[a_iT^i, b_jT^j]=\left( ia_i\delta_{i_0}(b_j)-jb_j\delta_{i_0}(a_i)\right) T^{i+j-1+i_0}+lower\ degree\ terms.\ee

For $a\in L$, write $\delta_{i_0}(a)=a'$.  
Then
\[ [z,w] = \sum_{i\leq n,j\leq m}[a_iT^i,b_jT^j] = (na_nb_m'-mb_ma_n')T^{n+m-1+i_0} + lower\ degree\   terms.\]
As $[z,w]=0$,  $na_nb_m'-mb_ma_n'=a_n^{m+1}b_m^{1-n}\delta_{i_0}(a_n^{-m}b_m^n)= 0$, i.e. $a_n^{-m}b_m^n \in \text{Ker}\delta_{i_0} = K$, then $b_m^n = \lambda a_n^m$
for some $\lambda \in K^{\times}$.
For $\widetilde{w} \in C(z)$ with $deg(\widetilde{w}) = m$ and leading term $\widetilde{b_m}T^m$,
one also has  $\widetilde{b_m}^n = \lambda'a_n^m$, where $\lambda'\in K^{\times}$. Then $\widetilde{b_m}^n = \mu b_m^n, \mu\in K^{\times}$.
As $K = \overline{K}$, $\widetilde{b_m}/b_m \in K^{\times}$.

Let $S = \{i \in \mathbb{Z}| \exists w \in C(z), deg(w) = i\}$. Then $S$ is a subgroup of $ \mathbb{Z}$.
As $ K[z,z^{-1}]\subseteq C(z)$, one has  $n\mathbb{Z} \subset S$, one has  $S = l\mathbb{Z}$ for some $l | n$ and $l>0$. As it is shown that for a given degree, the leading term of $x\in C(z)$ is unique
up to a factor of $K^\times$. Then, by the choice of $v$, we can choose $\lambda_{1}\in K^{\times}$ such that 
$w$ and $ \lambda_{1}v^{k_1}$ have the same leading term, where $k_1=m/l$. Let $w_1= w - \lambda_{1} v^{k_1}$,
then $w_1\in C(z)$ and $deg(w_{1}) < deg(w)$.
If $w_1\ne0$, repeat the process for $w_1$.
Continuing in this way, we can find infinitely many elements of $C(z)$,
\[w_s:= w_{s-1} - \lambda_{s} v^{k_s}, \lambda_{s}\in K^{\times}, s\geq 1, k_s\in \mathbb Z, where\quad w_0 := w,\]
such that $deg(w_1)>deg(w_2)>\cdots$ and $k_1>k_2>\cdots $.
Hence $w = \sum_{s\geq 1}\lambda_{s}v^{k_s}$.
Thus $C(z) \subseteq K((v^{-1}))$. The reverse inclusion is clear.
Hence one has  $C(z) = K((v^{{-1}}))$. 

The last statement of Case 2 follows from a simple degree argument.

%	Consider $S_1=\{n \in \mathbb{Z}| deg(w) = n, \exists w \in K((z^{-1}))\}$.	$S_1\subset S$ and $S = l\mathbb{Z}$, $S_1 = n\mathbb{Z}$, thus $[C(z) : K((z^{-1}))] = |\frac{deg(z)}{deg(v)}|$ if $deg(z) > 0$,
%	and $[C(z) : K((z))] = |\frac{deg(z)}{deg(v)}|$ if $deg(z) < 0$.
3. Recall that $z = a_0 + a_{-1}T^{-1} + \cdots\in D$ and $deg(z) =0$. Thus $ a_0\ne0$.

3.1. This is clear.

3.2. Assume $w = \sum_{j\leq m}b_jT^j\in C(z)$. By (\ref{D2}) one has  $m\delta_{i_0}(a_0)b_m = 0$,
as $K=\text{Ker}(\delta_{i_0})$ and $a_0\in L\setminus K$, one has $\delta_{i_0}(a_0)b_m\ne0$, thus $m$ must be $0$.
Then we define \[  \phi:C(z)\rightarrow L, w =b_0+b_{-1}T^{-1} + \cdots\mapsto b_0, \]
which  is a ring homomorphism. Next we will show that
$\phi$ is bijective, i.e.,  for any $ b_0\in L$, there exists some unique $w= b_0+b_{-1}T^{-1}+\cdots \in C(z)$ such that $[z,w] = 0$.

By (\ref{f24}) one has 
\be\label{D4} [z,w] =\sum_{i,j\le 0}[a_iT^i,b_jT^j]= \sum_{i,j\le 0} \sum_{k\ge 1}\sum_{ (j_1,\cdots,j_k )}[\binom{i}{k}a_i\cdot\delta_{j_1}\cdots\delta_{j_k}(b_j)-\binom{j}{k}b_j\cdot\delta_{j_1}\cdots \delta_{j_k}(a_i)]T^{j_1+\cdots+j_k-k+i+j} , %a,b\in L;m,n\in \mathbb Z
\ee
which can be written as $\sum_l f_l T^l$  by collecting the terms of the same power, $f_l$ is an algebraic expression of those $a_i,b_j$ and their images under (composition of) certain $\delta_k$.
As $[z,w] = 0$, one has  a series of equations by equating the coefficients $f_l$ of $T^l$ to be 0, from which one can solve $b_i$ for $i=-1,-2,\cdots,$ successively, as follows.
By (\ref{D4}), one has \[[z,w]= \left( b_{-1}\delta_{i_0}(a_0) - a_{-1}\delta_{i_0}(b_0) \right)T^{-2+i_0}+lower\ degree\ terms. \] As $\delta_{i_0}(a_0)\ne0$, there is a unique solution for $b_{-1}$ satisfying $b_{-1}\delta_{i_0}(a_0) - a_{-1}\delta_{i_0}(b_0)=0$. Assume $s<-1$ and  the solutions for $b_k$, $s< k\le-1$, have been found and are all unique. Now we consider $b_{s}$, who firstly occurs in the term $-sb_s\delta_{i_0}(a_0)T^{s-1+i_0} $ which is contained in the expansion of $[a_0,b_sT^s]$. Collecting all the terms of $T^{s-1+i_0} $ and equate the coefficient to be 0, one has 
\[-s\delta_{i_0}(a_0)b_s+d_s=0,\] where $d_s$  is an algebraic expression involves only $b_i, \delta_j(b_i),a_k, \delta_j(a_k), s+1\le i\le 0, s\le k\le 0$. As $char\ L=0$, one has $char\ K=0$. So $s\delta_{i_0}(a_0)\ne0$, thus there is a unique solution for $b_{s}$ in $L$. It follows by induction that for any $ b_0\in L$, there exists some unique $w=\sum_{s\le0}b_{s}T^{s}$ such that $[z,w] = 0$.
Thus $\phi$ is an isomorphism.
\ep

It is clear that $\widetilde{D}_{r,s}=K(\alpha )\left ( \left (T^{-1}; \widetilde{\delta }_{r, s }  \right )  \right )$, where $K$ is an algebraically closed field of characteristic 0, satisfies the condition of above proposition, thus  satisfies the ccc. 
 As $ \widehat{D}_{r,s}\subset \widetilde{D}_{r,s}$, $ \widehat{D}_{r,s}$ also satisfies the ccc. 

The result for $L_\textbf{r}((T^{-1};\widetilde{\delta}))$ analogous  to Proposition \ref{Dp1} can be obtained similarly. Thus $\check{D}_{r,s}=K(\beta)_\textbf{r}\left ( \left (T^{-1}; \widetilde{\eta}_{r, s }  \right )  \right )$,  where $K$ is an algebraically closed field of characteristic 0, also  satisfies the ccc. So one has

\begin{coro}
The division rings $\widetilde{D}_{r,s}, \widehat{D}_{r,s}$ and $\check{D}_{r,s} $ all satisfy the ccc.
\end{coro}

\end{document}